\documentclass[11pt]{article}
\usepackage{amsthm}
\usepackage{amsmath}
 \usepackage{a4}                          
\usepackage{amssymb}

\usepackage{xcolor}

\definecolor{collnk}{rgb}{.2,0.2,.6}
\definecolor{colcit}{rgb}{.1,0.5,.2}
\usepackage[colorlinks=true]{hyperref}  \hypersetup{citecolor=colcit,linkcolor=collnk}

\usepackage{graphicx}

\renewcommand{\a}{\alpha}
\renewcommand{\b}{\beta}
\newcommand{\C}{\mathbb{C}}
\newcommand{\R}{\mathbb{R}}
\newcommand{\N}{\mathbb{N}}
\newcommand{\D}{\mathbb{D}}
\newcommand{\Z}{\mathbb{Z}}

\newcommand{\Int}{\int_{-\infty}^\infty}
\DeclareMathOperator{\Log}{Log}
\DeclareMathOperator{\Arctan}{Arctan}

\DeclareMathOperator{\Rea}{Re}
 \DeclareMathOperator{\Ima}{Im}
\renewcommand{\Int}{\int_{-\infty}^\infty}
 
\newtheorem{thm}{Theorem}[section]
\newtheorem{lem}[thm]{Lemma}
\newtheorem{prop}[thm]{Proposition}
\newtheorem{rem}[thm]{Remark}
\newtheorem{cor}[thm]{Corollary}

\usepackage{subcaption}
\usepackage{placeins}

\usepackage[font=normalsize,skip=-4pt]{caption}

\title{A family of entire functions connecting the Bessel function $J_1$ and the Lambert $W$ function}
\author{Christian Berg, Eugenio Massa and Ana P. Peron}

\begin{document}

\maketitle


\begin{abstract}
Motivated by the problem of determining the values of $\a>0$ for which $f_\a(x)=e^\a - (1+1/x)^{\a x},\  x>0$ is a completely monotonic function, we combine Fourier analysis with complex analysis to find a family $\varphi_\a$, $\a>0$, of entire functions such that $f_\a(x) =\int_0^\infty e^{-sx}\varphi_\a(s)\,ds, \ x>0.$

We show that each function $\varphi_\a$ has an expansion in power series, whose coefficients are determined in terms of  Bell polynomials. This expansion leads to  several properties of the functions $\varphi_\a$, which turn out to be related to the well known  Bessel function $J_1$ and the Lambert $W$ function. 

On the other hand, by numerically evaluating the series expansion, we are able to show the
 behavior of $\varphi_\a$ as $\a$ increases from $0$ to $\infty$ and to obtain a very precise approximation of the largest $\a>0$ such that  $\varphi_\a(s)\geq0,\, s>0$, or equivalently, such that $f_\a$ is completely monotonic. 


\end{abstract}

\medskip
\section{Introduction and main results}\label{sec:intro}
A completely monotonic function is an infinitely differentiable function $f:\left]0,\infty\right[\to\R$ such that
$$
(-1)^nf^{(n)}(x)\ge 0,\quad x>0,\;n=0,1,\ldots,
$$
and a Bernstein function is an infinitely differentiable function $f:\left]0,\infty\right[\to\R$ such that
$f(x)\ge 0$ for $x>0$ and $f'$ is completely monotonic. Both classes of functions are treated in
\cite{B:F} and \cite{S:S:V}. The only completely monotonic functions, which are also Bernstein functions,  are the non-negative constant functions.

Let $\a>0, \b\in\R$. In \cite[p. 457]{A:B} it was proved that $(1+\a/x)^{x+\b}-e^\a$ is completely monotonic if and only if $\a\le 2\b$. This was sharpened in \cite{Q:N:C} to a proof that $(1+\a/x)^{x+\b}$ is logarithmically completely monotonic if and only if $\a\le 2\b$. Monotonicity properties of $(1+\a/x)^{x+\b}$ when $\a<0$ has been examined in
\cite{Q:L:G} and \cite{G:Q}.

For $\a>0$ define
\begin{equation*} 
f_\a(x)=e^\a - h_\a(x),\quad h_\a(x)=(1+1/x)^{\a x},\quad x>0.
\end{equation*}

In \cite[p. 458]{A:B} it was left as an open problem to determine the values of $\a>0$ for which
$e^\a-(1+\a/x)^x$ is completely monotonic or equivalently 
 $f_\a$ is completely monotonic. It was proved that $f_\a$ is completely monotonic for $0<\a\le 1$, and the question was, if $f_\a$ is completely monotonic for some $\a>1$. In \cite{B} the problem was given the equivalent formulation of determining the set of values  $\a>0$ such that $h_\a$ is a Bernstein function. It was noticed in \cite{B} that  $h_1$ is  a Bernstein function, because $f_1$ is completely monotonic, but $h_3$ is not a Bernstein function. Because of the fact that if $f$ is a Bernstein function, then so is $f^c$ for $0<c<1$, and the fact that the set of Bernstein functions is closed under pointwise convergence, the set in question is of the form $]0,\a^*]$, where $\a^*$ is an unknown number in the interval $[1,3[$. From graphs it looked probable that $\a^*>2$. 

In \cite{S:K:J} it was established numerically that $\a^*\approx 2.29965\,6443.$ This was done looking at monotonicity properties of high order derivatives of $f_\a$. More precisely, defining
$$
f(x,\a,n):=(-1)^nf_\a^{(n)}(x),\quad n=0,1,\ldots
$$
and letting $\a_n,x_n$ be determined as the "smallest positive solutions" to 
$$
f(x_n,\a_n,n+1)=f(x_n,\a_n,n+2)=0,
$$
then $\a_n$ decreases to $\a^*$. The estimate for $\a^*$ is then obtained from approximate values of $\a_n$ for
certain $n$ up to $n=10^5$.

In this paper we shall combine Fourier analysis with complex analysis to find a family of entire functions
 $\varphi_\a, \a>0$ such that
\begin{equation}\label{eq:Lap}
f_\a(x)=\int_0^\infty e^{-sx}\varphi_\a(s)\,ds, \quad x>0.
\end{equation} 
By a theorem of Bernstein, cf. \cite[p.160]{W}, this formula shows that $f_\a$ is completely monotonic if and only if $\varphi_\a(s)\ge 0$ for all $s>0$ and therefore $\a^*$ is determined as  the largest $\a>0$ such that $\varphi_\a(s)\ge 0$ for all $s>0$.

It turns out that our calculations leading to \eqref{eq:Lap} are valid for all complex $\a$, and for such $\a$   
we define
$$
f_\a(z)=e^\a-h_\a(z),\;
$$
\begin{equation}\label{eq:f1z}
 h_\a(z)=(1+1/z)^{\a z}:=\exp(\a z\Log(1+1/z)),\quad z\in\mathcal A,
\end{equation}
where $\mathcal A:=\C\setminus]-\infty,0]$ denotes the cut plane, and $\Log$ is the principal logarithm defined in $\mathcal A$.

The functions $\varphi_\a$ are given as contour integrals in the following theorem:
\begin{thm}\label{thm:intro1} Let $c>1, r>0$ be fixed,  and let $C(r,c)$ denote the rectangle with corners $-c\pm ir, \pm ir$ considered as a closed contour with positive orientation. Then for $\a\in\C$
\begin{equation}\label{eq:phi1}
\varphi_\a(s):=\frac{1}{2\pi i}\int_{C(r,c)}f_\a(z)e^{sz}\,dz,\quad s\in\C
\end{equation} 
is an entire function, which is independent of $c>1, r>0$, and \eqref{eq:Lap} holds for all $\a\in\C$.
Moreover $\varphi_\a(s)$ is bounded for $s\in [0,\infty[$ and tends to 0 for $s\to\infty$.
\end{thm}
Theorem~\ref{thm:intro1} is contained in Theorem~\ref{thm:cut} and in Theorem~\ref{thm:main1}. In particular, the formula \eqref{eq:Lap} is proved in Theorem~\ref{thm:main1}. 
 
The power series of the entire functions $\varphi_\a$ are given in the following theorem, depending on a remarkable sequence of polynomials:

\begin{thm}\label{thm:intro2} Let $(p_n)_{n\ge 0}$ denote the sequence of polynomials  defined by
\begin{equation}\label{eq:poly1} 
p_0(\a)=1,\quad p_1(\a)=\frac{\a}{2},\quad p_2(\a)=\frac{\a}{3}+\frac{\a^2}{8},\ldots,
\end{equation}
and in general
\begin{equation}\label{eq:poly-rec}
p_{n+1}(\a)=\frac{\a}{n+1}\sum_{k=0}^n \frac{k+1}{k+2} p_{n-k}(\a),\quad n\ge 0.
\end{equation}

For $\a\in \C$ 
\begin{equation}\label{eq:ent1}
\varphi_\a(s)=e^\a\sum_{n=0}^\infty(-1)^np_{n+1}(\a)\frac{s^n}{n!},\quad s\in\C.
\end{equation}
 In particular
 \begin{equation}\label{eq:phia(0)}
 \varphi_\a(0)=e^\a \a/2.
 \end{equation}
\end{thm}

Some properties of the polynomials $p_n$ are given in Proposition \ref{thm:p-bounds}, 
while  Theorem~\ref{thm:intro2} is proved in Section \ref{sec:fa}.

\medskip
It follows from \eqref{eq:ent1} that $(\a,s)\to \varphi_\a(s)$ is an entire function on $\C^2$, and $s\mapsto\varphi_\a(s)$ is not identically zero when $\a\neq 0$,  so it has at most countably many zeros $s\in\C$ which are all isolated. Furthermore when $\a>0$ then $\varphi_\a(s)>0$ for $s\le 0$. 

More results that can be deduced from \eqref{eq:ent1} are contained in the following.

\begin{thm}\label{thm:final} Consider the entire  function  $\C^2\to \C:(\a,s)\mapsto \varphi_\a(s)$.
	\begin{itemize}
	\item[(i)]	\begin{equation*} 
\lim_{\a\to 0}\frac{\varphi_\a(s)}{\a e^\a}=\sum_{n=0}^\infty \frac{(-1)^n}{n+2}\frac{s^n}{n!}
=\left\{ \begin{array}{ll} \displaystyle \frac{1-(1+s)e^{-s}}{s^2},\,& \mbox{when}\ s\neq 0, \\
\displaystyle\frac12,\,&\mbox{when}\ s=0,
\end{array}\right.  
\end{equation*}
uniformly for $s$ in compact subsets of the complex plane.

The limit function has no real zeros but infinitely many complex zeros $s=\xi_k=-W(k,-1/e)-1, k\in \Z \setminus\{-1,0\}$, where W(k,z) is the $k$'th branch of the Lambert $W$ function, see \cite{CGHJK}.
We have
$$\xi_1=\overline {\xi_{-2}}\approx 2.08884-7.46148 i,\qquad \xi_2=\overline {\xi_{-3}}\approx 2.66406-13.87905 i.$$
	\item[(ii)] Given $n\in\N$, $\varphi_\a$ has at least $n$ non-real zeros, when $|\a|$  is sufficiently small.

	\item[(iii)]
	\begin{equation*} 
	\lim_{|\a|\to\infty}\frac{{\varphi}_\a(s/\a)}{\a\,e^\a}=\frac{J_1(\sqrt{2s})}{\sqrt{2s}}
	\end{equation*}
	uniformly for $s$ in compact subsets of the complex plane, 	where $J_1$ is the Bessel function of order 1.	
	\item[(iv)]

Given $n\in\N$, $\varphi_\a$ has at least $n$ simple zeros $s_1(\a), s_2(\a),\ldots, s_n(\a)$ such that
$0<|s_1(\a)|<|s_2(\a)|<\ldots <|s_n(\a)|$ for $|\a|$ sufficiently large, and they satisfy
\begin{equation*} 
\lim_{|\a|\to\infty} \a s_k(\a)=\frac{j_k^2}{2} \qquad \mbox{for all}\; k\le n,
\end{equation*}     
where $0<j_1<j_2<\ldots$ are the positive zeros of $J_1$.

If in addition $\a>0$,  then $s_j(\a)>0,\ j=1,\ldots,n$.

	\item[(v)] For  $\a>0$,	the entire functions $\varphi_\a$ are of order one and type one. 
\end{itemize}
\end{thm}
	
\begin{rem}
{\rm 	It is worth observing that Property $(iv)$ above is an analytical proof of the existence of $\a^*$, in contrast with \cite{B}, where this was obtained  by computing  $f_3^{(4)}(0.4)<0$, which implies that $f_3$ cannot be completely monotonic.}
\end{rem}	
	The proof of Theorem \ref{thm:final} is given in Section \ref{sec:propphi}.

\medskip
If $(p_{n+1}(\a))_{n\ge 0}$ is a Stieltjes moment sequence, i.e., if there exists a positive measure $\sigma_\a$ on $[0,\infty[$ such that
\begin{equation}\label{eq:Sti}
p_{n+1}(\a)=\int_0^\infty x^n\,d\sigma_\a(x),\quad n\ge 0,
\end{equation}
then it is easy to see that
\begin{equation}\label{eq:Lapsigma}
\varphi_\a(s)=e^\a\int_0^\infty e^{-sx}\,d\sigma_\a(x),\quad s\in\C,
\end{equation}
and in particular $\varphi_\a(s)\ge 0$ for $s\ge 0$ and hence $0\le \a\le \a^*$.

However, this argument is only useful for $\a\le1$, in fact, the following holds. 
\begin{thm}\label{thm:St-cm} The following conditions are equivalent:
\begin{itemize}
\item[(i)] $(p_{n+1}(\a))_{n\ge 0}$ is a Stieltjes moment sequence.
\item[(ii)] $\varphi_\a$ is completely monotonic.
\item[(iii)] $0\le \a\le 1$.
\end{itemize}
If the equivalent conditions hold, then $\sigma_\a$ from \eqref{eq:Sti} is supported by $[0,1]$, and $(p_{n+1}(\a))_{n\ge 0}$ is a Hausdorff moment sequence.
\end{thm}
The proof is given in Section \ref{sec:St1}, where  we also find the measures $\sigma_\a$ for $0\le\a\le 1$
 (see the Equations \eqref{eq:St1} and \eqref{eq:St1_1}).

\medskip

For $\a>1$, on the other hand,  we show  in Theorem \ref{thm:alpha>1}  that the function 
$\varphi_\a$ can be decomposed as the sum of a completely monotonic function and a suitable contour integral (see Equation \eqref{eq:phi7}). 

 Even so,  we  have not been able to find an expression  which turns out useful in order to check if $\varphi_\a$ is non-negative  on $[0,\infty[$.
 As a consequence, for these values of $\a$, we have to rely on
numerical calculation.
For this purpose one can use the contour integral \eqref{eq:phi1}, 
but we prefer to use the power series \eqref{eq:ent1}, because of the following result.
\begin{thm}\label{thm:alt}  For $\a>0$, $n\ge 0$,  we know that $p_n(\alpha)> 0$ and 
	\begin{equation}\label{eq:ratbdd}
	\frac{p_{n+1}(\a)}{p_n(\a)}\le \widehat\a:=\left\{ \begin{array}{lll} 1,\,& \mbox{when}\, 0\le \a \le 1, \\
	2\a,\,&\mbox{when}\, 1 < \a<2, \\
	\a,\,&\mbox{when}\, 2\le \a.
	\end{array}\right.
	\end{equation}
	Then the series \eqref{eq:ent1} satisfies the Alternating Series Test for $n\ge \widehat\a s$,  which allows to obtain an error bound for the truncated series.
\end{thm}
The proof of Theorem~\ref{thm:alt} is given in Section \ref{sec:app}.

We summarize what can be seen from the numerical calculations in the following.
\begin{thm}[Numerical results]\label{thm:numric} $ $
\begin{itemize}
\item[(i)] $\a^*\approx 2.29965\, 64432\, 53461\, 30332$.	
\item[(ii)] For $0<\a<\a^*$ we have $\varphi_\a(s)>0$ for $s\ge 0$.	
\item[(iii)] $\varphi_{\a^*}(s)\ge 0$ for $s\ge 0$ and it  has a unique zero  of multiplicity two at $s^*\approx  5.27004\, 87522\, 76132\, 37103$.	
\item[(iv)] For $\a^*<\a$, $\varphi_\a$  has a finite number of positive zeros $0<s_1(\a)<s_2(\a)<\ldots<s_n(\a)$ which are all simple with the exception that the last can be double. 	 
\item[(v)]  $s_1(\a)$ is a simple zero with $\varphi_\a'(s_1(\a))<0$, moreover
	$s_1(\a)$ is a decreasing function on $]\a^*,\infty[$.
\end{itemize}
\end{thm}

Below we present several graphs that support the claims in Theorem \ref{thm:numric}.




\newcommand{\wid}{11cm}

\begin{figure}[th]
	\begin{center}
		\fbox{\includegraphics[width=\wid,keepaspectratio]{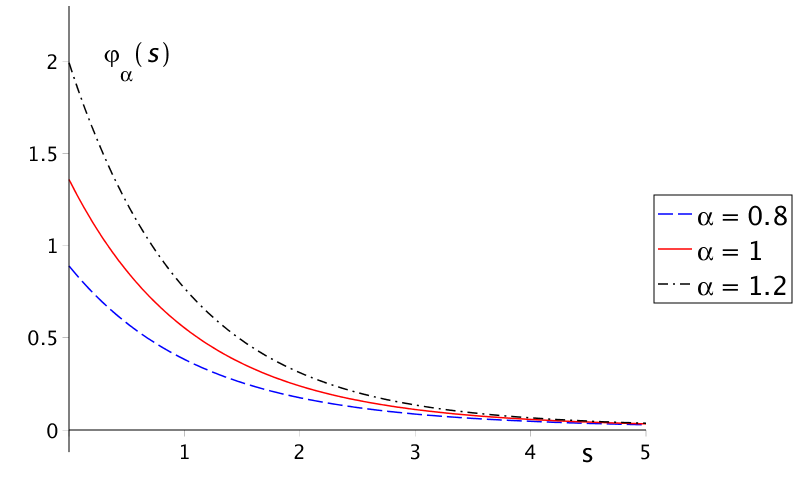}}
	\end{center}
	\caption{$\varphi_\a$ near $\a=1$}
	\label{fig_phi1}
\end{figure}
\begin{figure}[th]
	\begin{center}
	\fbox{\includegraphics[width=\wid,keepaspectratio]{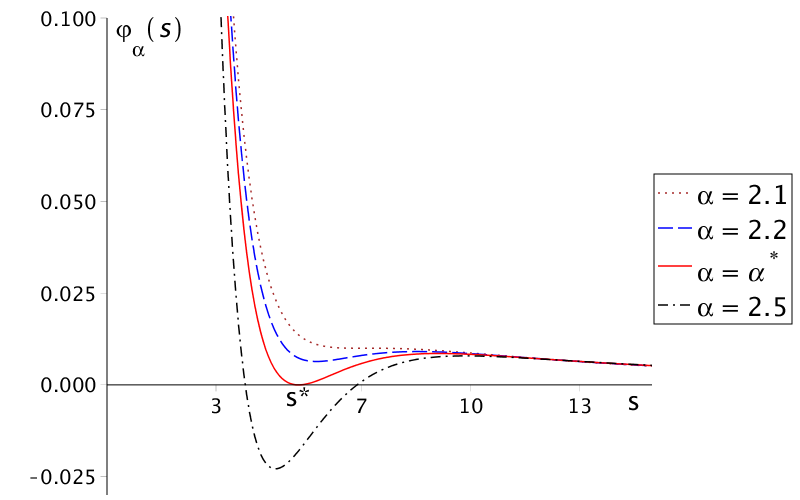}}
	\end{center}
	\caption{$\varphi_\a$ below, at, and above $\a^*$}
	\label{fig_phi*}
\end{figure}
\begin{figure}[th]
	\begin{center}
		\fbox{\includegraphics[width=\wid,keepaspectratio]{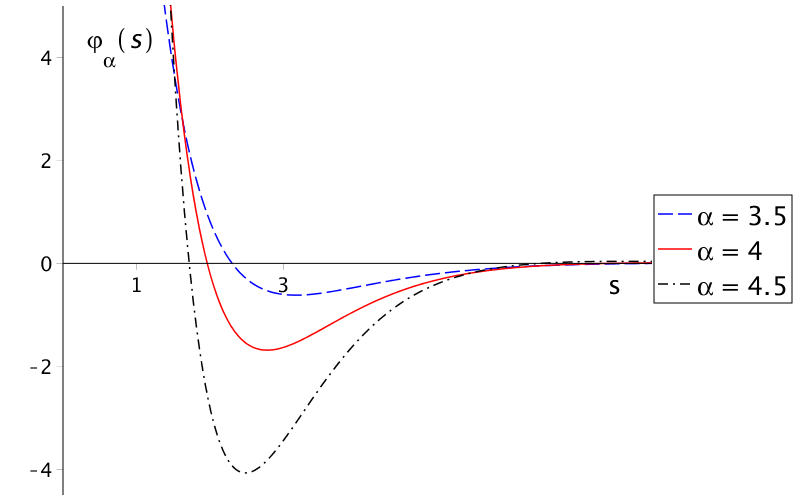}}
	\end{center}
	\caption{$s_1(\a)$ decreases}
	\label{fig_s_dec}
\end{figure}
\begin{figure}[th]
	\begin{center}
		\fbox{\includegraphics[width=\wid,keepaspectratio]{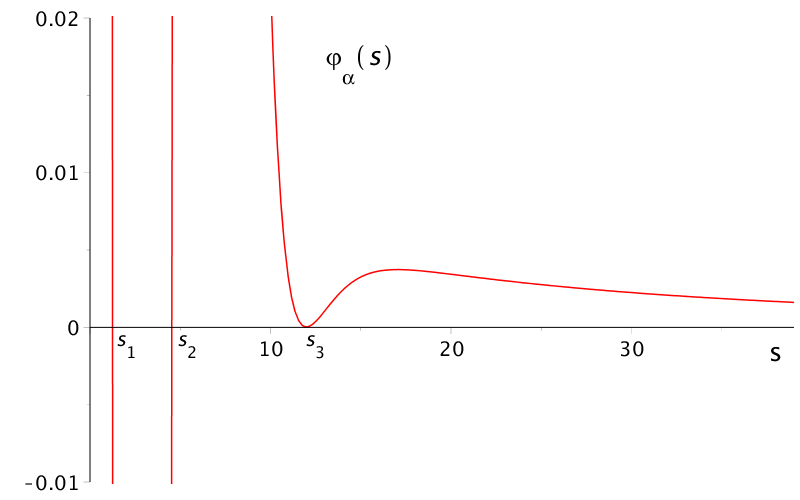}}
	\end{center}
	\caption{Formation of a third zero ($\a\approx 5.988$) }
	\label{fig_2ndzer}
\end{figure}
\begin{figure}[th]
	\begin{center}
		\fbox{\includegraphics[width=\wid,keepaspectratio]{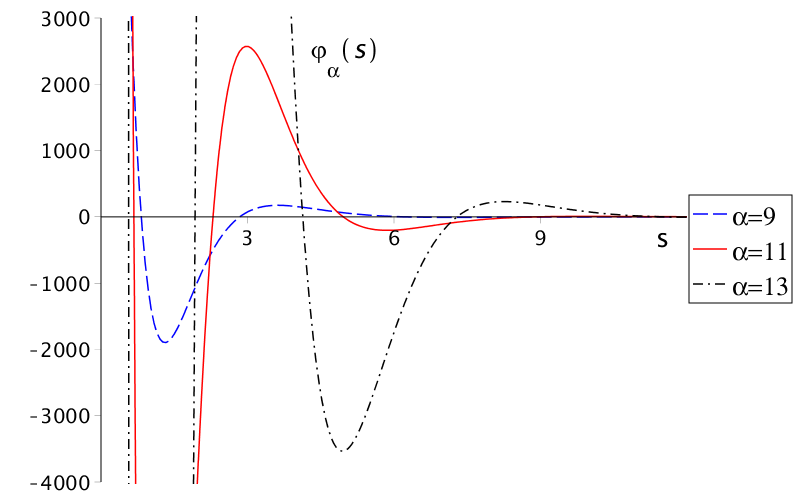}}
	\end{center}
	\caption{Increasing oscillations at larger $\a$ }
	\label{fig_phiosc}
\end{figure}
\begin{figure}[th]
	\begin{center}
		\fbox{\includegraphics[width=\wid,keepaspectratio]{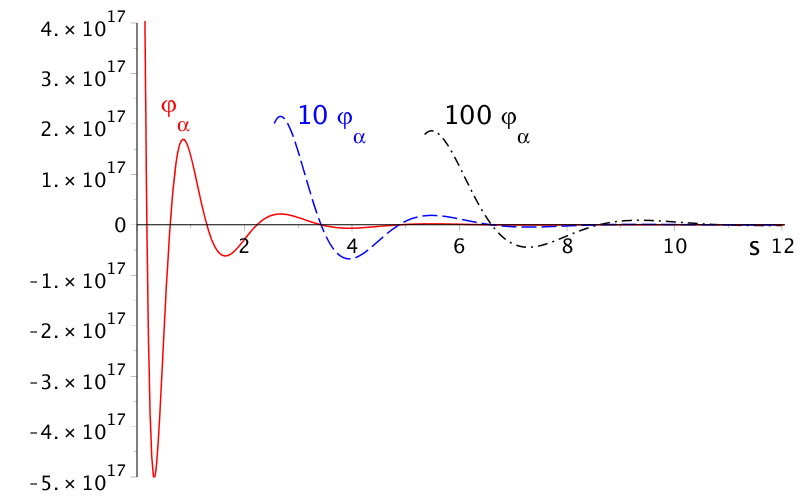}}	
	\end{center}
	\caption{Oscillations at  $\a=40$ }
	\label{fig_phiosc2}
\end{figure}




In Figure \ref{fig_phi1} the graph of ${\varphi}_\a$ is sketched for the  values  $\a=0.8,\ 1,\ 1.2$: in these cases ${\varphi}_\a$ is strictly positive (see Property $(ii)$) and, for $\a=0.8,\ 1$, also completely monotonic, as stated in Theorem \ref{thm:St-cm}. 

In Figure \ref{fig_phi*} one can see the graph of  ${\varphi}_\a$ for the value  $\a=\a^*$ given in $(i)$, where it presents a unique zero $s^*$, which is also a global minimum, as described in $(iii)$. On the other hand,  ${\varphi}_\a$ is still strictly positive for $\a<\a^*$ and has a  region of negative values between two simple zeros for $\a>\a^*$. 

As $\a$ increases, the first zero $s_1(\a)$ decreases as described in $(v)$ (Figure \ref{fig_s_dec}). For $\a\approx 5.988$ (Figure \ref{fig_2ndzer}) a new double zero appears on the right of $s_1$ and $s_2$, and  then more and more oscillations appear by the same mechanism (Figure \ref{fig_phiosc}). In Figure \ref{fig_phiosc2}, for instance, $\varphi_{40}$ is represented with 3 different scales, and one can see at least 10 zeros.

The graphs are obtained in Maple by truncating the series \eqref{eq:ent1}, taking into consideration Theorem \ref{thm:alt}. The approximated values of $\a^*,\, s^*$ given in Theorem \ref{thm:numric}-$(i,\,iii)$ are also obtained from 
the truncated series by seeking the minimal value  $\a^*$ for which $\varphi_{\a^*}$ is zero at some $s^*>0$. The approximation for $s^*$ is then improved using the fact that $\varphi_{\a^*}'(s^*)=0$.

\section{The family $\varphi_\a, \a\in\C$ and the polynomials $p_n(\a)$}\label{sec:fa}
From \eqref{eq:f1z} it is easy to see that 
\begin{equation}\label{eq:fa(0)}
\lim_{z\in\mathcal A, z\to 0}f_\a(z)=e^\a-1.
\end{equation}
Moreover, $h_0(x)=1$ for all $x>0$, so $h_0$ is both completely monotonic and a Bernstein function, while
$h_{-\a}(x)$ is completely monotonic for all $\a>0$ because of \cite[Proposition 9.2]{B:F}, where we use that
$$
x\log(1+1/x)=\int_0^1\left(1-\frac{u}{u+x}\right)\,du,\quad x>0
$$
is a Bernstein function.  In particular $h_\a$ is not a Bernstein function when $\a<0$.
This means that the set of $\a\in\R$ such that $h_\a$ is a Bernstein function is the closed interval $[0,\a^*]$.

For an open set $G\subseteq \C$ we denote by $\mathcal H(G)$ the set of holomorphic functions
defined in $G$.

 Clearly $f_\a\in\mathcal H(\mathcal A)$, but we shall see that $f_\a$ extends to a holomorphic function in $\C\setminus[-1,0]$.
In fact, for $z\neq 0, |z|<1$ we have
\begin{equation}\label{eq:g}
g_\a(z):=e^\a-\exp(\a z^{-1}\Log(1+z))=e^\a -\exp(\a(1-z/2+z^2/3-\cdots)),
\end{equation}
so defining $g_\a(0)=0$, we see that $g_\a\in\mathcal H(\D)$, where
$$
\D:=\{z\in\C:|z|<1\}.
$$
Now $f_\a(z):=g_\a(1/z)$ for $|z|>1$ yields a holomorphic extension of $f_\a$ to $\C\setminus [-1,0]$.

In the next two results we obtain suitable power series expansions for $g_\a$, $f_\a$ and $h_\a$. 
\begin{prop}\label{thm:power-g}
The power series of $g_\a\in\mathcal H(\D)$ given  by \eqref{eq:g} can be written
\begin{equation*} 
g_\a(z)=e^\a\sum_{n=1}^\infty (-1)^{n-1}p_n(\a)z^n,
\end{equation*}
where $(p_n(\a))_{n\ge 0}$ is the sequence of polynomials defined in \eqref{eq:poly1} and \eqref{eq:poly-rec}.
\end{prop}
\begin{proof} We use the formula
\begin{equation*} 
\exp\left(\sum_{k=1}^\infty \frac{a_k}{k!}z^k\right)=\sum_{n=0}^\infty \frac{B_n(a_1,\ldots,a_n)}{n!}z^n,
\end{equation*}
where $B_n$ are the  exponential Bell partition polynomials, cf. \cite[Section 11.2]{Ch}. It is known that 
$$
B_0=1,\quad B_1(a_1)=a_1,\quad B_2(a_1,a_2)=a_1^2+a_2,
$$
and in general we have the recursion formula
$$
B_{n+1}(a_1,\ldots,a_{n+1})=\sum_{k=0}^n\binom{n}{k}B_{n-k}(a_1,\ldots,a_{n-k})a_{k+1}.
$$
Defining $a_k=(-1)^k\a k!/(k+1),\;k\ge 1$, this gives for $z\in\D$
\begin{eqnarray*}
g_\a(z)&=&e^\a-e^\a\exp\left(\sum_{k=1}^\infty(-1)^k\a\frac{z^k}{k+1}\right)\\
&=&-e^\a\sum_{n=1}^\infty\frac{B_n(a_1,\ldots,a_n)}{n!}z^n=e^\a\sum_{n=1}^\infty
(-1)^{n-1}p_n(\a)z^n,
\end{eqnarray*}
where we have defined
\begin{equation}\label{eq:pnBn}
p_n(\a):=(-1)^n\frac{B_n(a_1,\ldots,a_n)}{n!}.
\end{equation}
We see by induction that $p_n(\a)$ is a polynomial in $\a$ of degree $n$ such that \eqref{eq:poly1}
holds, and the recursion \eqref{eq:poly-rec} follows like this
\begin{eqnarray*}
p_{n+1}(\a)&=& (-1)^{n+1}\frac{B_{n+1}(a_1,\ldots,a_{n+1})}{(n+1)!}\\
&=&\frac{1}{n+1}\sum_{k=0}^n \frac{1}{k!}(-1)^{n-k}\frac{B_{n-k}(a_1,\ldots,a_{n-k})}{(n-k)!}
(-1)^{k+1}a_{k+1}\\
&=&\frac{\a}{n+1}\sum_{k=0}^n\frac{k+1}{k+2}p_{n-k}(\a).
\end{eqnarray*}
\end{proof}

\begin{cor}\label{thm:laurent-f} For $|z|>1$ we have the Laurent expansions
\begin{equation}\label{eq:laurent-f}
f_\a(z)=e^\a\sum_{n=1}^\infty (-1)^{n-1}\frac{p_n(\a)}{z^n},\quad h_\a(z)=
e^\a\sum_{n=0}^\infty (-1)^{n}\frac{p_n(\a)}{z^n}
\end{equation} 
and in particular
\begin{equation}\label{eq:f-infty}
f_\a(z)=(\a/2)e^\a z^{-1}+\mathcal O(|z|^{-2}), \quad |z|>1, |z|\to\infty.
\end{equation}
\end{cor}

In the following lemma we study the restriction of the function $f_\a$ to the imaginary axis. 
\begin{lem}\label{thm:L2} 
Let $\a\in\C\setminus\{0\}$. As a function of $y\in\R$
\begin{equation*} 
F_\a(y):= \left\{ \begin{array}{lll} f_\a(iy)=e^\a -\left(1+y^{-2}\right)^{i\a y/2}\exp(\a y\Arctan(1/y)),&y\neq 0,\\e^\a-1, &y= 0,
\end{array}\right.
\end{equation*}
is continuous and tends to 0 for $|y|\to\infty$. It belongs to
$L^2(\R)\setminus L^1(\R)$. 	
\end{lem}
\begin{proof} We have for $y\neq 0$
$$
\exp(i\a y\Log(1-i/y))=\exp(i\a y[\log\sqrt{1+y^{-2}}-i\Arctan(1/y)]),
$$
where $\Arctan:\R\to ]-\pi/2,\pi/2[$ is the inverse of $\tan$. The continuity of $F_\a$ for $y=0$ follows, and the behavior at $\pm\infty$ including the integrability properties follows from Corollary~\ref{thm:laurent-f}.
\end{proof}

By Plancherel's Theorem  $F_\a$ is the Fourier-Plancherel transform of another $L^2$-function $G_\a$:
\begin{equation}\label{eq:Plan1}
f_\a(iy)=\lim_{R\to\infty}\int_{-R}^R G_\a(s)e^{-iys}\,ds,\quad y\in \R,
\end{equation}
where the limit is in $L^2(\R)$,
and by the inversion theorem $G_\a$ is given as
\begin{equation}\label{eq:Plan}
G_\a(s)=\lim_{R\to\infty}\frac{1}{2\pi}\int_{-R}^R f_\a(iy)e^{iys}\,dy,\quad s\in\R,
\end{equation} 
where again the limit is in $L^2(\R)$. For certain sequences $R_n\to \infty$ we also know that
\begin{equation}\label{eq:Plan2}
\lim_{n\to\infty} \frac{1}{2\pi}\int_{-R_n}^{R_n} f_\a(iy)e^{iys}\,dy=G_\a(s)
\end{equation}
for almost all $s\in\R$.
Furthermore,
\begin{equation}\label{eq:Pl}
\frac{1}{2\pi}\Int |f_\a(iy)|^2\,dy=\Int |G_\a(s)|^2\,ds.
\end{equation}

\begin{rem} \label{thm:remark1}
 {\rm The above formulas \eqref{eq:Plan1}-\eqref{eq:Pl} hold trivially for $\a=0$ with $F_0=G_0=0$.}
\end{rem}
\medskip

\begin{lem}\label{thm:s<0}
We have $G_\a(s)=0$ for $s<0$.
\end{lem}
\begin{proof} Let $C_R$ denote the half-circle with radius $R$
$$
C_R=[-iR,iR]\cup \{Re^{it}:-\pi/2\le t\le \pi/2\},
$$
considered as a positively oriented closed contour.

Since $f_\a(z)e^{sz}$ is holomorphic in $\Rea z>0$ with a continuous extension to the closed right half-plane bounded by the vertical line $i\R$,
we have by Cauchy's integral theorem 
$$
\int_{-R}^R f_\a(iy)e^{isy}i\,dy=\int_{-\pi/2}^{\pi/2} f_\a(Re^{it})e^{sRe^{it}}Re^{it}i\,dt.
$$
The absolute value of the integrand to the right is by \eqref{eq:f-infty} bounded by
$$
\frac{C}{R}e^{sR\cos(t)}R
$$
for a suitable $C>0$ depending on $\a$. If we assume $s<0$, then $sR\cos t\to-\infty$ for $R\to\infty$ when
$-\pi/2<t<\pi/2$, so the integral to the right tends to 0 by dominated convergence. Using \eqref{eq:Plan2} we now see that $G_\a(s)=0$ for almost all $s<0$. Since $G_\a$ is an equivalence class of square integrable functions, we can assume that $G_\a(s)=0$ for $s<0$.
\end{proof}

Exploiting the holomorphy of $f_\a$ in $\C\setminus[-1,0]$ we can  prove part of Theorem \ref{thm:intro1}, which is contained in the following.
\begin{thm}\label{thm:cut} 
 For $\a\in\C$, the function $\varphi_\a$ defined in \eqref{eq:phi1}  is an entire function, which is independent of $c>1, r>0$.

Moreover, the function
\begin{equation}\label{eq:phi2}
s\mapsto \left\{ \begin{array}{ll} \varphi_\a(s),\, &\mbox{when}\ 0\le s<\infty,  \\
0,\,&\mbox{when}\, -\infty <s<0,
\end{array}\right.
\end{equation} 
is equal to $G_\a(s)$ for almost all $s\in\R$.
\end{thm}

\begin{rem}\label{thm:remark2}
{\rm In the following we denote the function given by \eqref{eq:phi2} as $G_\a$. }
\end{rem}

\begin{proof} It is clear that the function $\varphi_\a$ from \eqref{eq:phi1} is entire,
and also that it is independent of $c>1, r>0$ by Cauchy's integral theorem.

Let $R>c$ and consider the following three positively oriented closed contours: two quarter circles with radius $R$
\begin{eqnarray*}
 T_+(R)&=&\{x+ir:x=-R\ldots 0\}\\
&&\cup\;\{iy:y=r\ldots r+R\}\cup\{ir+Re^{i\theta}:\theta=\pi/2\ldots\pi\}\\
 T_-(R)&=&\{x-ir:x=0\ldots -R\}\\
&&\cup\;\{-ir+Re^{i\theta}:\theta=\pi\ldots 3\pi/2\}\cup\{iy:y=-r-R\ldots -r\}
\end{eqnarray*}
and a rectangle $Q(R)$ with corners $\{-R\pm ir, -c\pm ir\}$.
By Cauchy's integral theorem we have
$$
\frac{1}{2\pi i}\int_{T_\pm(R)} f_\a(z)e^{sz}\,dz=\frac{1}{2\pi i}\int_{Q(R)} f_\a(z)e^{sz}\,dz=
0.
$$
Adding these three integrals to the contour integral in \eqref{eq:phi1}  yields
\begin{eqnarray*}
\varphi_\a(s)&=&\frac{1}{2\pi i}\int_{C(r,c)} f_\a(z)e^{sz}\,dz\\
&=&\frac{1}{2\pi i}\int_{-i(r+R)}^{i(r+R)} f_\a(z)e^{sz}\,dz + \frac{1}{2\pi i}\int_{L(R)}f_\a(z)e^{sz}\,dz,
\end{eqnarray*}
where $L(R)$ is the following contour
$$
\{ir+Re^{i\theta}:\theta=\pi/2\ldots\pi\}\cup\{-R+iy:y=r\ldots -r\}\cup\{-ir+Re^{i\theta}:\theta=\pi\ldots 3\pi/2\}.
$$
For $s>0$ we have
$$
\lim_{R\to\infty} \frac{1}{2\pi i}\int_{L(R)}f_\a(z)e^{sz}\,dz=0.
$$
In fact, for $z\in L(R)$ we have $|z|\ge R$, hence $|f_\a(z)|\le C/R$ for suitable $C>0$ by \eqref{eq:f-infty}. This gives
$$
\left|\frac{1}{2\pi i}\int_{L(R)}f_\a(z)e^{sz}\,dz\right|\le \frac{C}{2\pi R}\left[2re^{-sR}+R\int_{\pi/2}^{3\pi/2} e^{sR\cos\theta}\,d\theta\right],
$$
which tends to 0 for $R\to\infty$ because $\cos\theta<0$ for 
$\pi/2<\theta<3\pi/2$.
 The function
$$
I(R)(s):=\frac{1}{2\pi i}\int_{-i(r+R)}^{i(r+R)} f_\a(z)e^{sz}\,dz
$$
converges to $G_\a(s)$ in $L^2(\R)$ for $R\to\infty$ by \eqref{eq:Plan}, so for a suitable sequence $R_n\to\infty$
we know that $I(R_n)(s)$ converges to $G_\a(s)$ for almost all $s\in\R$. It follows that $\varphi_\a(s)=G_\a(s)$ for almost all $s>0$.  
 Since apriori we only know that $G_\a$ is an equivalence class of square integrable functions, we can use formula \eqref{eq:phi2} as a representative of $G_\a$.
\end{proof}

At this point we are able to prove Theorem \ref{thm:intro2}, that is, to obtain the power series expansion of $\varphi_\a$.
\begin{proof}[Proof of Theorem \ref{thm:intro2}] From \eqref{eq:phi1} and the compactness of the contour $C(r,c)$ we get
$$
\varphi_\a(s)=\sum_{n=0}^\infty \frac{s^n}{n!}\frac{1}{2\pi i}\int_{C(r,c)}f_\a(z)z^n\,dz.
$$
Using that $f_\a(z)z^n$ is holomorphic outside $[-1,0]$, we can replace the contour $C(r,c)$ by the circle $|z|=R_0$, where $R_0>\sqrt{c^2+r^2}>1$. We next use the Laurent expansion \eqref{eq:laurent-f} and get
\begin{eqnarray*}
\frac{1}{2\pi i}\int_{C(r,c)}f_\a(z)z^n\,dz&=&\sum_{k=1}^\infty e^\a (-1)^{k-1}p_k(\a)\frac{1}{2\pi i}\int_{|z|=R_0} z^{n-k}\,dz\\
&=&e^\a(-1)^n p_{n+1}(\a),
\end{eqnarray*}
which shows \eqref{eq:ent1}.
\end{proof}

In the following proposition we list several properties of the polynomials $p_n$ that appear in the  series \eqref{eq:ent1}. We prove below the Properties $(i)$ through $(iv)$, in particular, Properties $(i)$ through $(iii)$  will be needed in Theorem \ref{thm:main1}, in order to conclude the proof of Theorem \ref{thm:intro1}.
The remaining properties rely partially on the results of Section \ref{sec:St1} and will be proved in  Section \ref{sec:app}. 
 See also Remark \ref{thm:referee} and Equation \eqref{eq:referee1} in  the Appendix, for an alternative expression for the polynomials $p_n$, based on Stirling numbers.

\begin{prop}\label{thm:p-bounds} The polynomials $p_n$ from Theorem~\ref{thm:intro2} satisfy Theorem \ref{thm:alt} and 
\begin{itemize}
\item[(i)] $\displaystyle p_n(\a)=\sum_{k=1}^n c_{n,k}\a^k, \quad n\ge 1$,	
	where 
$c_{n,k}>0$ and 
\\ \hspace*{1cm}{$\displaystyle c_{n,1}=\frac1{(n+1)},\; c_{n,n}=\frac{1}{2^n\,n!}\,.$}	
\item[(ii)] $|p_n(\a)|\le p_n(|\a|),\quad \a\in\C$.	
\item[(iii)] $0\le p_n(\a)\le \left\{ \begin{array}{lll} 1,\, &\mbox{when }\, 0\le \a \le 1, \\
	\a^n,\,&\mbox{when }\, 1 \le \a, \\
	n,\,&\mbox{when }\,0\le \a\le 2, n\ge 1.
	\end{array}\right.$	
\item[(iv)] For $\a,\b\in\C$ and $n\ge 0$ we have the addition formula
	$$
	p_n(\a+\b)=\sum_{k=0}^n p_k(\a)p_{n-k}(\b). 
	$$	
\item[(v)] The sequence  $(p_n(\a))_{ n\ge 0}$ is $ \left\{ \begin{array}{lll} \text{strictly decreasing, }&\text{ for $0<\a\le 1$, } \\
	\text{increasing, }&\text{ for $2\le \a$. }
	\end{array}\right.$ 	
\item[(vi)] $
	\displaystyle\lim_{n\to\infty} p_n(\a)=\left\{ \begin{array}{ll} 0,\, &\mbox{when }\, 0<\a <1, \\
	e^{-1},\,&\mbox{when }\, \a=1, \\
	\infty,\,&\mbox{when }\, \a>1.
	\end{array}\right.
	$
\item[(vii)] $\displaystyle \lim_{n\to\infty} \frac{p_{n+1}(\a)}{p_n(\a)}=\lim_{n\to\infty} \root{n}\of {p_n(\a)}=1,\, \mbox{when}\ \a>0$.
\end{itemize}
\end{prop}

\begin{proof}[Proof of Proposition \ref{thm:p-bounds}: (i)  through (iv)]
	Property $(i)$ follows easily by induction using the recursion \eqref{eq:poly-rec}, while Property $(ii)$ follows  because the coefficients of $p_n$ are non-negative. 
	
	Property $(iii)$ follows by induction, as described below. The two first inequalities hold for $n=0$, and assuming the assertion for $k\le n$ we get by 
	\eqref{eq:poly-rec}, for $0\le\a\le 1$
	$$
	p_{n+1}(\a)\le\frac{1}{n+1}\sum_{k=0}^n \frac{k+1}{k+2}\le 1,
	$$
	and for $1\le \a$
	$$
	p_{n+1}(\a)\le\frac{\a}{n+1}\sum_{k=0}^n \frac{k+1}{k+2}\a^{n-k} \le \a^{n+1}.
	$$
	The last inequality holds for $n=1$, and by induction using $p_0(\a)=1$ and  $p_k(\a)\le k$ for $1\le k\le n$ we find, for $n\ge 1$,
	$$
	p_{n+1}(\a)\le \frac{\a}{n+1}\left(1+\sum_{k=0}^{n-1} (n-k)\right)=\frac{\a}{n+1}\left[1+\frac{n(n+1)}2\right]\le \frac{2}{n+1} +n\le n+1.
	$$

	To see $(iv)$ we notice that by \eqref{eq:f1z} we get
	$$
	h_{\a+\b}(z)=h_\a(z)h_\b(z),\quad z\in\mathcal A,
	$$
	and by \eqref{eq:laurent-f}
	$$
	h_\a(-1/s)=e^\a\sum_{n=0}^\infty p_n(\a)s^n,\quad 0<|s|<1.
	$$
	This implies that
	\begin{equation}\label{eq:powerprod}
	\sum_{n=0}^\infty p_n(\a+\b) s^n=\sum_{n=0}^\infty p_n(\a)s^n \sum_{m=0}^\infty p_m(\b)s^m,\quad 0<|s|<1,
	\end{equation}
	which clearly holds for $s=0$. Multiplying the absolutely convergent power series for $|s|<1$ in \eqref{eq:powerprod}, we get the addition formula.
\end{proof}

By Theorem~\ref{thm:intro2} and Proposition~\ref{thm:p-bounds}-$(ii,\,iii)$ we obtain the following.
\begin{cor}\label{thm:remarkvarphi} The entire functions $\varphi_\a$ satisfy for $s,\a\in\C$
$$
|\varphi_\a(s)|\le \left\{ \begin{array}{ll} |e^\a|e^{|s|}\,,\, &\mbox{when }\,  |\a| \le 1 \\
|\a||e^\a|e^{|\a s|}\,,\,&\mbox{when }\, 1 \le |\a|.
\end{array}\right.
$$
\end{cor}
The order and type of $\varphi_\a$ when $\a>0$ are given in Theorem~\ref{thm:final}-$(v)$.

We can finally conclude the proof of Theorem \ref{thm:intro1}, as a consequence of the following.   
\begin{thm}\label{thm:main1} For any $\a\in\C$,  $\varphi_\a(s)$ is bounded for $s\in [0,\infty[$ and tends to 0 for $s\to\infty$.
The following formula holds
\begin{equation}\label{eq:plan2}
f_\a(z)=\int_0^\infty \varphi_\a(s) e^{-sz}\,ds,\quad \Rea{z}> 0.
\end{equation}
\end{thm}
\begin{proof}
Let us write \eqref{eq:phi1} in another way. Introducing
$$
I_1(r,s):=\frac{1}{2\pi}\int_{-r}^r f_\a(iy)e^{isy}\,dy,\quad 
I_3(r,s):=\frac{1}{2\pi}\int_{-r}^r f_\a(-c+iy)e^{s(-c+iy)}\,dy,
$$
and
$$
I_2(r,s):=\frac{1}{2\pi i}\int_{-c}^0 f_\a(x+ir)e^{s(x+ir)}\,dx,\quad
I_4(r,s):=\frac{1}{2\pi i}\int_{-c}^0 f_\a(x-ir)e^{s(x-ir)}\,dx,
$$
we have for $s\ge 0$
\begin{equation}\label{eq:phi21}
\varphi_\a(s)=I_1(r,s)-I_3(r,s) -I_2(r,s) +I_4(r,s).
\end{equation}
We see by  Riemann-Lebesgue's lemma that $\lim_{s\to\infty}I_1(r,s)=0$. Furthermore, 
$$
|I_3(r,s)|\le \frac{r}{\pi}e^{-sc}\max\{|f_\a(-c+iy)|:|y|\le r\}
$$
and
$$
|I_2(r,s)|, |I_4(r,s)|\le \frac{1}{2\pi}\max\{|f_\a(x \pm ir)|: -c\le x\le 0\}\int_{-c}^0 e^{sx}\,dx
$$ 
show that
$$
\lim_{s\to\infty} |I_j(r,s)|=0,\quad j=2,3,4.
$$  
By \eqref{eq:phi21} it follows that $\lim_{s\to\infty}\varphi_\a(s)=0$. This property together with the 
continuity of $\varphi_\a$ imply that $\varphi_\a$ is bounded on $[0,\infty[$.

To prove the formula \eqref{eq:plan2} we note that the right-hand side is  holomorphic for $\Rea z>0$, and so is the left-hand side. 

For $s\ge 0$ we have by Proposition~\ref{thm:p-bounds}-$(ii,\,iii)$
\begin{equation}\label{eq:maj}
\left|\sum_{n=0}^N(-1)^np_{n+1}(\a)\frac{s^n}{n!}\right|
\le
\sum_{n=0}^N p_{n+1}(|\a|)\frac{s^n}{n!}
\le \left\{\begin{array}{ll} e^s,\;&\mbox{if}\; |\a|\le 1\\
|\a|e^{|\a|s},\;&\mbox{if}\; |\a|\ge 1.
\end{array}\right.
\end{equation}
For $x>0$ we have
\begin{eqnarray*}
&&\int_0^\infty e^{\a} \left(\sum_{n=0}^N (-1)^np_{n+1}(\a)\frac{s^n}{n!}\right)e^{-sx}\,ds\\
&=&e^\a\sum_{n=0}^N (-1)^n p_{n+1}(\a)\int_0^\infty \frac{s^n}{n!}e^{-sx}\,ds=
e^{\a}\sum_{n=0}^N (-1)^n\frac{p_{n+1}(\a)}{x^{n+1}}.
\end{eqnarray*}
Assume now  $x>\max(1,|\a|)$. For $N\to\infty$ the last expression converges to $f_\a(x)$ by Corollary~\ref{thm:laurent-f}.
The integrand  in the first expression converges for each $s\ge 0$ to $\varphi_\a(s)e^{-sx}$
with an integrable majorant because of \eqref{eq:maj}, so by Lebesgue's Theorem on dominated convergence, we  get
$$
\int_0^\infty \varphi_\a(s)e^{-sx}\,ds=f_\a(x),\quad x>\max(1,|\a|).
$$
This is enough to conclude \eqref{eq:plan2}.
\end{proof}

As discussed in the Introduction, we can now state the following important result. 

\begin{thm}\label{thm:main2} For $\a\in\C$, $f_\a$ is completely monotonic if and only if $\varphi_\a(s)\ge 0$ for $s\ge 0$.
In the affirmative case $\varphi_\a$ is integrable on $[0,\infty[$ and
\begin{equation}\label{eq:intphi}
\lim_{x\to 0^+}f_\a(x)=e^\a-1=\int_0^\infty \varphi_\a(s)\,ds.
\end{equation}
Moreover, in this case \eqref{eq:plan2} holds for $\Rea{z}\ge 0$ and
\begin{equation}\label{eq:Bnstein}
h_\a(z)=1+\int_0^\infty (1-e^{-sz})\varphi_\a(s)\,ds,\quad \Rea{z}\ge 0,z\neq 0.
\end{equation}
\end{thm}
\begin{proof} The first assertion follows from Bernstein's characterization of completely monotonic functions as Laplace transforms of positive measures. Equation \eqref{eq:intphi} follows from \eqref{eq:fa(0)}, \eqref{eq:plan2} and the monotonicity theorem of Lebesgue. 

When $\varphi_\a$ is integrable over $[0,\infty[$, the right-hand side of \eqref{eq:plan2} is continuous in the half-plane $\Rea{z}\ge 0$, and since $f_\a$ is also continuous there, we see
that \eqref{eq:plan2} holds for $\Rea{z}\ge 0$. The Equation \eqref{eq:Bnstein} follows easily from \eqref{eq:intphi}.
\end{proof}

\section{The cases $0<\a\le 1$	}\label{sec:St1}
As mentioned in Section \ref{sec:intro},  it was proved in \cite{A:B} that $f_\a$ is completely monotonic for $0<\a\le 1$ and equivalently $\varphi_\a$ is non-negative on $[0,\infty[$ for these values of $\a$.  We shall use the previous results to give a new proof of this. We recall that a function $f:\left]0,\infty\right[\to\R$ is called a Stieltjes function, if it has the form
\begin{equation}\label{eq:St}
f(s)=a+\int_0^\infty \frac{d\mu(t)}{s+t}, \quad s>0,
\end{equation}
where $a\ge 0$ and $\mu$ is a positive measure on $[0,\infty[$. A Stieltjes function is completely monotonic but the converse is not true.  For more information about Stieltjes functions see \cite{B:F} and \cite{B2008}.

We have the following result.
\begin{thm}\label{thm:Stigen}The function $f_\a$ is a Stieltjes function for $0\le\a\le 1$, but not for $\a>1$.
\end{thm}

The cases $0<\a<1$, $\a=1$ and $\a>1$ are treated separately in Theorem~\ref{thm:Sta1},
Corollary~\ref{thm:repa=1} and Proposition~\ref{thm:nonSti}.

\begin{thm}\label{thm:Sta1} For $0<\a<1$ we have
\begin{equation}\label{eq:St1}
\varphi_\a(s)=\frac{1}{\pi}\int_0^1 (x/(1-x))^{\a x}\sin(\a\pi x) e^{-sx}\,dx,\quad s\ge 0,
\end{equation}
and
\begin{equation}\label{eq:St2}
f_\a(z)=\frac{1}{\pi}\int_0^1 \frac{(x/(1-x))^{\a x}\sin(\a\pi x)}{x+z} \,dx,\quad z\in\C\setminus [-1,0].
\end{equation}
\end{thm}
\begin{proof}
Assume $0<\a<1$ and let $c, r$ in the contour from Theorem~\ref{thm:cut} be chosen such that $1<c<2, \a c<1$ and $0<r<1$.

Let now  $r\to 0$ in \eqref{eq:phi21}. The first two terms tend to 0. Using that $\a$ is real we can write
$$
I_2(r,s)-I_4(r,s)=\frac{1}{\pi}\int_{-c}^0 e^{sx}\Ima\{f_\a(x+ir)e^{isr}\} \,dx,
$$
and replacing $x$ by $-x$ in this expression, we get
\begin{eqnarray*}
\varphi_\a(s)&=&-\lim_{r\to 0}\frac{1}{\pi}\int_0^{c} e^{-sx}\Ima\{f_\a(-x+ir)e^{isr}\}\,dx\\
&=& \lim_{r\to 0}\frac{1}{\pi}\int_0^{c}  e^{-sx}\Ima \left\{\exp \left[\a(-x+ir)\Log\left(1+\frac{1}{-x+ir}\right) +irs\right]\right\}\,dx.
\end{eqnarray*}
We have 
$$
\Log\left(1+\frac{1}{-x+ir}\right)=K(x,r)-i\theta(x,r)
$$
where
$$
K(x,r)=\frac12\log \frac{(1-x)^2+r^2}{x^2+r^2},\quad  \cot\theta(x,r)=\frac{x(x-1)+r^2}r
$$
and $\theta(x,r)\in \left]0,\pi\right[$.

We therefore have (leaving out the arguments in $K(x,r), \theta(x,r)$ to simplify notation)
\begin{eqnarray*}
J_r(x) &:=& \Ima \left\{\exp \left[\a(-x+ir)\Log\left(1+\frac{1}{-x+ir}\right) +irs\right]\right\}\\
&=& \exp[\a(-xK+r\theta)]\sin[\a(rK+x\theta)+rs]\\
&=& \left(\frac{x^2+r^2}{(1-x)^2+r^2}\right)^{(\a x)/2}\exp(\a r\theta)\sin[\a(rK+x\theta)+rs],
\end{eqnarray*}
and hence
$$
\lim_{r\to 0}J_r(x)=\left\{ \begin{array}{lll} (x/(1-x))^{\a x}\sin(\a\pi x),\,& \mbox{when}\, 0<x<1,  \\
\infty,\, &\mbox{when}\, x=1,\\
0,\,&\mbox{when}\, 1<x<c.
\end{array}\right.
$$

We have the following inequalities for $0<x<c$, using that $|x-1|<1$, 
\begin{eqnarray*}
|J_r(x)| &\le& \left(\frac{x^2+r^2}{(1-x)^2+r^2}\right)^{(\a x)/2}\exp(\a \pi)\le 
\frac{(c^2+1)^{(\a c)/2}\exp(\a\pi)}{|1-x|^{\a c}},
\end{eqnarray*}
and since  $\a c<1$, the last expression is an integrable majorant over $]0, c[$.
By Lebesgue's Theorem we therefore get
\begin{equation}\label{eq:phi3}
\varphi_\a(s)=\frac{1}{\pi}\int_0^1 \left(x/(1-x)\right)^{\a x}\sin(\a\pi x) e^{-sx}\,dx,
\end{equation}
so $\varphi_\a(s)>0$ for $s\ge 0$ and $0<\a<1$.
 
Inserting \eqref{eq:phi3} in \eqref{eq:plan2} we get
\begin{equation}\label{eq:plan3}
f_\a(z)=\frac{1}{\pi}\int_0^1 \frac{\left(x/(1-x)\right)^{\a x}\sin(\a\pi x)}{x+z}\,dx,\qquad \text{for $\Rea{z}>0$.} 
\end{equation}
By the identity theorem for holomorphic functions \eqref{eq:plan3} holds for $z\notin [-1,0]$. 
\end{proof}

Equation \eqref{eq:phi3} shows that $\varphi_\a$ is completely monotonic for $0<\a<1$. For $\a\to 1^-$ we get that $\varphi_1$ is completely monotonic and in particular non-negative, and by \cite[Section 14.12]{B:F} we get that $f_1$ is a Stieltjes function.

To find the representations of $\varphi_1$ and $f_1$ in analogy with \eqref{eq:phi3} and \eqref{eq:plan3}, it turns out not to be correct to replace $\a$ by 1 in these formulas. 

\newcommand{\g}{u}\newcommand{\h}{w}
Let us introduce the notation
\begin{equation*} 
\g(\a,x)=\left(x/(1-x)\right)^{\a x}\sin(\a\pi x),\quad 0<\alpha\le 1, 0\le x<1.
\end{equation*}
Clearly $\g(\a,x)\ge 0$, and $\g(1,x)$ is seen to be bounded by $\pi$, while
$$
\lim_{x\to 1} \g(\a,x)=\infty, \quad 0<\a<1.
$$

\begin{prop}\label{thm:weak} For $0<\a<1$ define
\begin{equation*} 
\h(\a,x)=\g(\a,x)-\g(1,x), \quad 0\le x<1.
\end{equation*}
Then for any $\phi\in C([0,1])$ we have
\begin{equation}\label{eq:weak2}
\lim_{\a\to 1^-}\frac{1}{\pi}\int_0^1 \h(\a,x)\phi(x)\,dx=\phi(1).
\end{equation}
\end{prop}
\begin{proof} We need the following partial results:

\medskip

{\bf Step 1:} $\lim_{\a\to 1^-} \h(\a,x)=0$ for $0\le x<1$, uniformly for $x\in[0,1-\delta]$ for any $0<\delta<1$.

This is clear.
\medskip

{\bf Step 2:}
$$
 \lim_{\a\to 1^-}\frac{1}{\pi}\int_0^1 \h(\a,x)\,dx=1.
$$

To see this, note that by \eqref{eq:St1} and \eqref{eq:phia(0)}
\begin{equation}\label{eq:I1} 
\frac{1}{\pi}\int_0^1 \g(\a,x)\,dx=\varphi_\a(0)=\frac\a2\,e^\a
\end{equation}
and
\begin{equation}\label{eq:I2}
\frac{1}{\pi}\int_0^1 \g(1,x)\,dx=
\frac{1}{\pi}\int_0^1 \left(x/(1-x)\right)^{x}\sin(\pi x)\,dx= \frac e2-1,
\end{equation}
see \cite[Lemma 2, p. 4]{A:B}.
Therefore
$$
\frac{1}{\pi}\int_0^1 \h(\a,x)\,dx=\frac\a2\,e^\a -\frac e2 +1,
$$
which has limit 1 for $\a\to 1$, proving Step 2.

\medskip
Let now $\phi\in C([0,1])$. Let $\varepsilon>0$ be given and by continuity choose $x_1<1$ such that $|\phi(x)-\phi(1)|<\varepsilon$ for $x_1\le x\le 1$. We can then write
\begin{eqnarray*}
&&\frac{1}{\pi}\int_0^1 \h(\a,x)\phi(x)\,dx-\phi(1)\\
&=&\frac{1}{\pi}\int_0^{x_1} \h(\a,x)\phi(x)\,dx+ 
\frac{1}{\pi}\int_{x_1}^1 \h(\a,x)(\phi(x)-\phi(1))\,dx\\
&+&\phi(1)\left(\frac{1}{\pi}\int_{x_1}^1 \h(\a,x)\,dx-1\right):=\sum_{j=1}^3 T_j(\a),
\end{eqnarray*}
and hence
$$
\left|\frac{1}{\pi}\int_0^1 \h(\a,x)\phi(x)\,dx-\phi(1)\right|\le \sum_{j=1}^3 |T_j(\a)|.
$$
By { Step  1} we know that $|T_1(\a)|\to 0$ for $\a\to 1^-$. Furthermore, by \eqref{eq:I1} and \eqref{eq:I2} we find

\begin{eqnarray*}
|T_2(\a)|&\le& \frac{1}{\pi}\int_{x_1}^1 |\h(\a,x)||\phi(x)-\phi(1)|\,dx\le \frac{\varepsilon}{\pi}
\int_{x_1}^1|\h(\a,x)|\,dx\\
&\le & \frac{\varepsilon}{\pi}\int_0^1 (\g(\a,x)+\g(1,x))\,dx =\varepsilon\left(\frac\a2\,e^\a+\frac e2\,-1\right)\le \varepsilon(e-1).
\end{eqnarray*}
Finally, $|T_3(\a)|$ tends to 
0 for $\a\to 1^-$ because
$$
|T_3(\a)|=|\phi(1)|\left|\frac{1}{\pi}\int_{0}^1 \h(\a,x)\,dx-1-\frac{1}{\pi}\int_{0}^{x_1} \h(\a,x)\,dx\right|,
$$
and we then use { Step 1} and { Step 2}. 

In total we get
$$
\limsup_{\a\to 1^-}\left|\frac{1}{\pi}\int_0^1 \h(\a,x)\phi(x)\,dx-\phi(1)\right|\le \varepsilon(e-1),
$$
and \eqref{eq:weak2} follows. 
\end{proof}

Applying the above result to the continuous functions $\phi(x)=e^{-sx}$ and $\phi(x)=(x+z)^{-1}$ for $z\notin [-1,0]$ we get

\begin{cor}\label{thm:repa=1}
\begin{equation}\label{eq:St1_1}
\varphi_1(s)=e^{-s}+\frac{1}{\pi}\int_0^1 \left(x/(1-x)\right)^{x}\sin(\pi x) e^{-sx}\,dx,\quad s\ge 0.
\end{equation}
\begin{equation*} 
f_1(z)=
\frac{1}{z+1}+ \frac{1}{\pi}\int_0^1 \frac{\left(x/(1-x)\right)^{x}\sin(\pi x)}{x+z}\,dx, \quad
z\notin [-1,0].
\end{equation*}
\end{cor}
 $ $
\begin{prop}\label{thm:nonSti} The function $f_\a$ is not a Stieltjes function when $\a>1.$
\end{prop}
\begin{proof} By \eqref{eq:St} a non-constant Stieltjes function $f$ has an extension to a holomorphic function in $\mathcal A$ satisfying
\begin{equation}\label{eq:Stho}
\Ima{f(z)}<0 \quad \mbox{for}\quad \Ima{z}>0,
\end{equation}
because for $z=x+iy, y>0$ we have
$$
\Ima{f(z)}=-y\int_0^\infty \frac{d\mu(t)}{|z+t|^2} <0.
$$
For $\a>1$ let $0<x<1$ be chosen such that $1<\a x<2$. For $y=r>0$ we have
$$
\Ima{f_\a(-x+ir)}=-\Ima\left\{\exp\left[\a(-x+ir)\Log\left(1+\frac{1}{-x+ir}\right)\right]\right\},
$$
and proceeding as in the proof of Theorem~\ref{thm:Sta1} we  get
$$
\lim_{r\to 0} \Ima{f_\a(-x+ir)}=-(x/(1-x))^{\a x}\sin(\a\pi x)>0.
$$
This shows that $\Ima{f_\a(-x+ir)}>0$ for $r>0$ sufficiently small. By \eqref{eq:Stho} this shows that $f_\a$ is not a Stieltjes function when $\a>1$.
\end{proof}

Using the formulas for $\varphi_\a$ in Theorem~\ref{thm:Sta1} and Corollary~\ref{thm:repa=1} we can prove that the sequence $(p_{n+1}(\a))_{n\ge 0}$  is a Hausdorff moment sequence, i.e., the moment sequence of a positive measure on $[0,1]$.

\begin{thm}\label{thm:Haus} For $0<\a<1$ we have
\begin{equation}\label{eq:Haus1}
p_{n+1}(\a)=\frac{e^{-\a}}{\pi}\int_0^1 (x/(1-x))^{\a x}\sin(\a\pi x) x^n\,dx,\quad n\ge 0,
\end{equation}
while for $\a=1$
\begin{equation}\label{eq:Haus2}
p_{n+1}(1)=e^{-1} +\frac{e^{-1}}{\pi}\int_0^1 (x/(1-x))^{ x}\sin(\pi x) x^n\,dx,\quad n\ge 0.
\end{equation}
\end{thm}
\begin{proof} Inserting the power series for $e^{-sx}$ in Equation \eqref{eq:St1} and interchanging summation and integration, we get the power series expansion for $\varphi_\a$. Compared with \eqref{eq:ent1} this yields \eqref{eq:Haus1}.

To get the case $\a=1$ we can proceed similarly with the formula for $\varphi_1$ in Corollary~\ref{thm:repa=1}, or we can apply Proposition~\ref{thm:weak} to $\phi(x)=x^n$.
\end{proof}
\noindent See Remark \ref{thm:n=-1} in the Appendix, for a proof that the sequence $(p_{n}(\a))_{n\ge 0}$  is also a Hausdorff moment sequence.

We can now prove the equivalence of the three conditions in Theorem~\ref{thm:St-cm}. 
\begin{proof}[Proof of Theorem~\ref{thm:St-cm}]$ $\\
"$(i) \Rightarrow (ii)$" If $(p_{n+1}(\a))_{n\ge 0}$ is a Stieltjes moment sequence, i.e., \eqref{eq:Sti} holds for a positive measure $\sigma_\a$ on $[0,\infty[$, then $\a=2p_1(\a)=2\int_0^\infty\,d\sigma_\a(x)\ge 0$.
Without loss of generality we can assume $\a>0$. By Proposition~\ref{thm:p-bounds}-$(iii)$ we know that  $p_n(\a)\le l^n,\;n\ge 0$, where $l=\max(1,\alpha)$,  which implies that $\sigma_\a$ is supported by the interval $[0,l]$. By \eqref{eq:ent1}  we then get
$$
\varphi_\a(s)=e^\a\sum_{n=0}^\infty (-1)^n\frac{s^n}{n!}\int_0^l x^n\,d\sigma_\a(x)
=e^\a \int_0^l e^{-sx}\,d\sigma_\a(x),
$$
which shows that $\varphi_\a$ is completely monotonic.

"$(ii) \Rightarrow (iii)$" If $\varphi_\a$ is completely monotonic, hence of the form
$$
\varphi_\a(s)=\int_0^\infty e^{-ts}\,d\mu(t)
$$
for a positive measure $\mu$, we get, using \eqref{eq:ent1},
$$
(-1)^n\varphi_\a^{(n)}(0)=\int_0^\infty t^n\,d\mu(t)=e^\a p_{n+1}(\a)\geq0,\quad n\geq 0
$$
but this is only possible if $\a\geq0$. Furthermore, 
$$
f_\a(x)=\int_0^\infty e^{-xs}\varphi_\a(s)\,ds=\int_0^\infty\frac{d\mu(t)}{x+t},\quad x>0,
$$
so $f_\a$ is a Stieltjes function  and hence $\a\le 1$ by Theorem~\ref{thm:Stigen}.

"$(iii) \Rightarrow(i)$" follows from Theorem~\ref{thm:Haus}. 
\end{proof}

\section{The case $\a>1$}\label{sec:a>1}
In the previous section we were able to express the functions $\varphi_\a$ with $0\leq\a\leq1$ as in Equation \eqref{eq:Lapsigma}, proving that they are nonnegative on $[0,\infty[.$  The purpose of this section is to show that, for $1<\a$, we can still find a component in $\varphi_\a$ analogous to  \eqref{eq:Lapsigma}, but  a correcting term needs to be added, which is given by a contour integral on a suitable circle that goes around the singularity $-1$: see Equations \eqref{eq:phi7} and \eqref{eq:Phi}.  

\medskip

For $a\in\C$ and $r>0$ we denote by  $\partial D(a,r)$ the positively oriented circle with center $a$ and radius $r$.

Let $\a>1$ be fixed, and let $0<\varepsilon<1-1/\a$. We consider the closed positively oriented contour $T(\a,\varepsilon)$ starting at $i\varepsilon$, then moving left along the horizontal line $x+i\varepsilon$ till it cuts the circle $\partial D(-1,1-1/\a)$ at a point denoted $x(\varepsilon)+i\varepsilon$. We then move along the circle till we reach the complex conjugate point $x(\varepsilon)-i\varepsilon$ (passing $-2+1/\a$ on the way), and then we move along the horizontal line $x-i\varepsilon$ till we reach $-i\varepsilon$, which is connected to $i\varepsilon$ via  the vertical segment $iy, y\in[-\varepsilon,\varepsilon]$.  

The contour $T(\a,\varepsilon)$ can replace the contour $C(r,c)$ of Theorem~\ref{thm:cut} so we have
\begin{equation}\label{eq:phi5}
\varphi_\a(s)=\frac{1}{2\pi i}\int_{T(\a,\varepsilon)}f_\a(z)e^{sz}\,dz,\quad s\ge 0.
\end{equation}

We shall now obtain a new expression for $\varphi_\a$ by letting $\varepsilon$ tend to 0. Note that $\lim_{\varepsilon\to 0} x(\varepsilon)=-1/\a$.

This leads to the following result.

\begin{thm}\label{thm:alpha>1} For $\a>1$ we have 
\begin{eqnarray}\label{eq:phi7}
\varphi_\a(s)=\frac{1}{\pi}\int_0^{1/\a}(x/(1-x))^{\a x}\sin(\a\pi x)e^{-sx}\,dx-\Phi(\a,s),
 \end{eqnarray}
where
\begin{equation}\label{eq:Phi}
\Phi(\a,s):=\frac{1}{2\pi i}\int_{\partial D(-1,1-1/\a)}h_{\a}(z)e^{sz}\,dz
\end{equation}
and $h_\a(z)$ is given in \eqref{eq:f1z}.
The first term on the right-hand side of  \eqref{eq:phi7} is a completely monotonic function.
\end{thm}
\begin{proof}
Letting $\varepsilon\to 0$ in \eqref{eq:phi5}, we note that  the contribution from $iy, y\in[-\varepsilon,\varepsilon]$ tends to $0$, and we get
$$
\varphi_\a(s)=\Phi_1(\a,s)-\Phi_2(\a,s),\quad s\ge 0,
$$
where
\begin{equation}\label{eq:Phi1}
\Phi_1(\a,s):=\frac{1}{2\pi i}\int_{\partial D(-1,1-1/\a)}f_\a(z)e^{sz}\,dz=-\Phi(\a,s)
\end{equation}
and
\begin{equation*} 
\Phi_2(\a,s):=
\lim_{\varepsilon\to 0}\frac{1}{\pi}\int_{0}^{-x(\varepsilon)}\Ima\{ f_\a(-x+i\varepsilon) e^{s(-x+i\varepsilon)}\}\,dx.
\end{equation*}
In \eqref{eq:Phi1} we 
used that the term $e^\a e^{sz}$ has integral 0 over the circle, because it is an entire function of $z$. 
We further get
\begin{eqnarray*}
\lefteqn{\Phi_2(\a,s)=}\\
&-&\lim_{\varepsilon\to 0}\frac{1}{\pi}\int_{0}^{-x(\varepsilon)}\Ima\{\exp[\a(-x+i\varepsilon)\Log(1+1/(-x+i\varepsilon))+s(-x+i\varepsilon)]\}\,dx\\
&=& -\frac{1}{\pi}\int_0^{1/\a}(x/(1-x))^{\a x}\sin(\a\pi x)e^{-sx}\,dx,
\end{eqnarray*}
 where we have used the same technique as in the proof of Theorem~\ref{thm:Sta1}. 
This gives formula \eqref{eq:phi7}.
\end{proof}

\section{Properties of the sequences $(p_n(\a))_{n\ge 0}$. }\label{sec:app}
This section is devoted to the proof of the remaining properties of the polynomials $p_n$, which were stated in   Proposition~\ref{thm:p-bounds} and Theorem \ref{thm:alt}.

\begin{proof}[Proof of Proposition~\ref{thm:p-bounds}-(v,\,vi)]
In the case $0<\a\leq1$ Properties $(v)$ and $(vi)$ follow directly from the formulas \eqref{eq:Haus1} and \eqref{eq:Haus2}. 

From \eqref{eq:poly-rec} we estimate, for $\a>0$,

\begin{eqnarray}\label{eq:pnp_pn}
	p_{n+1}(\a)&=&\frac{\a}{n+1}\left[\frac12 p_n(\a)+\sum_{k=1}^{n}\frac{k+1}{k+2}p_{n-k}(\a)
	\right]\nonumber\\
	&=&\frac{\a}{n+1}\left[\frac12 p_n(\a)+\sum_{k=0}^{n-1}\frac{k+2}{k+3}p_{n-1-k}(\a)\right]\nonumber\\
	&\ge&\frac{\a}{n+1}\left[\frac12 p_n(\a)+\sum_{k=0}^{n-1}\frac{k+1}{k+2}p_{n-1-k}(\a)\right]\nonumber\\
	&=&\frac{\a}{n+1}\left[\frac12 p_n(\a)+\frac{n}{\a}p_n(\a)\right]= \frac{\a/2+n}{n+1}p_n(\a)
\end{eqnarray}
	and for $2\le\a$ this proves $(v)$. 

Property $(vi)$ for $1<\a$ will be proved in  Proposition \ref{thm:toinfty}.
\end{proof}

In order to study further the sequence $(p_n(\a))_{n\ge 0}$ it is useful to introduce  the sequence of the mean-values
\begin{equation*} 
M_n(\a):=\frac{1}{n+1}\sum_{k=0}^n p_k(\a),\quad n\ge 0,
\end{equation*} 
which satisfies the recursion 
\begin{equation}\label{eq:mean2}
M_n(\a)=\frac{nM_{n-1}(\a)+p_n(\a)}{n+1}.
\end{equation}

Note that by \eqref{eq:poly-rec}
\begin{equation}\label{eq:mean1}
\frac{\a}{2}M_n(\a)\le p_{n+1}(\a)< \a M_n(\a),\quad \a>0,\, n\ge 0.
\end{equation}

 Observe that  for any $\a>0$ we can estimate, for  $n\ge k_0\ge1$,
\begin{eqnarray}\label{eq:fund}
p_{n+1}(\a)&=&\frac{\a}{n+1}\left[\sum_{k=0}^{k_0-1}\frac{k+1}{k+2}p_{n-k}(\a)+
\sum_{k=k_0}^{n}\frac{k+1}{k+2}p_{n-k}(\a)\right]\nonumber\\
&\ge& \frac{\a}{n+1}\sum_{k=0}^{k_0-1}\frac{k+1}{k+2}p_{n-k}(\a) + \frac{k_0+1}{k_0+2}\frac{\a}{n+1}
\sum_{k=k_0}^n p_{n-k}(\a) \nonumber\\
&=&\a\frac{k_0+1}{k_0+2}M_n(\a)-\frac{\a}{n+1}\sum_{k=0}^{k_0-1}\left(\frac{k_0+1}{k_0+2}-\frac{k+1}{k+2}\right)p_{n-k}(\a),
\end{eqnarray}
where  the last sum in \eqref{eq:fund} is positive.

\begin{prop}\label{thm:toinfty} For $\a>1$ we have
$$
\lim_{n\to\infty} M_n(\a)=\infty,
\qquad\qquad
\lim_{n\to\infty} p_n(\a)=\infty.
$$
\end{prop}
\begin{proof}
{\bf The case $\a>2$.}

In this case $p_{n+1}(\a)>M_n(\a)$ by \eqref{eq:mean1}, and then  \eqref{eq:mean2} implies that $M_{n+1}(\a)> M_n(\a)$, so $M_n(\a)$ increases to $C\le \infty$. We claim that $C=\infty$ and the proposition is proved because of \eqref{eq:mean1}. We shall see that the assumption $C<\infty$ leads to a contradiction.
We choose a sufficiently small $\delta>0$ so that
$$
\frac{\a}{2}(C-\delta) > C+\delta,
$$ 
and next $n_0\in\N$ so that $M_n(\a)>C-\delta$ for $n\ge n_0$. We then get
$$
p_{n+1}(\a)\ge \frac{\a}{2} M_n(\a)>\frac{\a}{2}(C-\delta)>C+\delta,\quad n\ge n_0,
$$
which leads to a contradiction, since it implies that  $M_n(\a)$ will  eventually exceed $C$.

\medskip
{\bf The case $1<\a\le 2$.}

We proceed in steps:

{\bf Step 1: The sequence $(p_n(\a))_{n\ge 0}$ is unbounded.}

Assume for contradiction that $p_n(\a)\le B<\infty$ for all $n$. Then also $M_n(\a)\le B$, and since 
$p_n(\a)\ge p_n(1)\ge 1/e$ by Proposition~\ref{thm:p-bounds}-$(i,\,v,\,vi)$, we also get $M_n(\a)\ge 1/e$.

For a given $\a>1$ we choose the smallest $k_0\in \N$ so that
\begin{equation}\label{eq:ak0}
\a\frac{k_0+1}{k_0+2}>1
\end{equation}
and  $\varepsilon>0$ such that 
\begin{equation*} 
\a\frac{k_0+1}{k_0+2}>1+\varepsilon.
\end{equation*}

From \eqref{eq:fund} we then get
\begin{eqnarray*}
p_{n+1}(\a)&> &(1+\varepsilon)M_n(\a)-\frac{\a}{n+1}\sum_{k=0}^{k_0-1}\left(\frac{k_0+1}{k_0+2}-\frac{k+1}{k+2}\right)B \\
&\ge & \left(1+\frac\varepsilon2\right)M_n(\a) +\frac\varepsilon{2e} -\frac{\a}{n+1}\sum_{k=0}^{k_0-1}\left(\frac{k_0+1}{k_0+2}-\frac{k+1}{k+2}\right)B, 
\end{eqnarray*}
and since the last term tends to 0 for $n\to\infty$, we get 
$$
p_{n+1}(\a) \ge  \left(1+\frac\varepsilon2\right)M_n(\a)>M_n(\a), \quad n\ge n_0,
$$
where $n_0\in\N$ is sufficiently large. Therefore $M_{n+1}(\a)>M_n(\a)$ for  $n\ge n_0$ and finally
$C:=\lim_{n\to\infty}M_n(\a)$ exists and $C\le B$.

Let now $\delta>0$ be so small that by \eqref{eq:ak0} 
\begin{equation*} 
\a\frac{k_0+1}{k_0+2}(C-\delta)>C+\delta,
\end{equation*}
and let $n_1>n_0$ be so large that $M_n(\a)>C-\delta$ for $n\ge n_1$.  By \eqref{eq:fund} we get for $n\ge \max(k_0,n_1)$
\begin{eqnarray*}
p_{n+1}(\a) &> &   \a\frac{k_0+1}{k_0+2}(C-\delta) -\frac{\a}{n+1}\sum_{k=0}^{k_0-1}\left(\frac{k_0+1}{k_0+2}-\frac{k+1}{k+2}\right)p_{n-k}(\a)\\
&>& C+\delta -\frac{\a}{n+1}\sum_{k=0}^{k_0-1}\left(\frac{k_0+1}{k_0+2}-\frac{k+1}{k+2}\right)B,
\end{eqnarray*}
hence
$$
p_{n+1}(\a)> C+\delta/2,\quad n\ge n_2,
$$
where $n_2$ is sufficiently large. Therefore $M_n(\a)$ will eventually be larger than $C$, which is a contradiction, and we have proved Step 1.

\medskip
{\bf Step 2: The sequence $(M_n(\a))_{n\ge 0}$ is eventually strictly increasing.}

Once this is proved, we know that $\lim_{n\to\infty} M_n(\a)=\infty$ for otherwise $(M_n(\a))$ is a bounded
sequence, and so is $(p_n(\a))$ by \eqref{eq:mean1}, and this contradicts Step 1.

To see Step 2 we note that by Step 1 there exist indices $n_1<n_2<\cdots$ such that $(p_{n_j+1}(\a))_{ j\ge 1}$ is strictly increasing to infinity.  By \eqref{eq:mean1} we get that 
$$
\lim_{j\to\infty} M_{n_j}(\a)=\infty.
$$
We now use that $1<\a\le 2$ and hence $p_n(\a)\le n$ for $n\ge 1$ by Proposition~\ref{thm:p-bounds}-$(iii)$. From \eqref{eq:fund} we then get for $n\ge k_0$
\begin{eqnarray}\label{eq:fundl}
p_{n+1}(\a) &\ge& \a\frac{k_0+1}{k_0+2}M_n(\a)-\frac{\a}{n+1}\sum_{k=0}^{k_0-1}\left(\frac{k_0+1}{k_0+2}-\frac{k+1}{k+2}\right)(n-k)\nonumber\\
&>&  \a\frac{k_0+1}{k_0+2}M_n(\a)-\a\sum_{k=0}^{k_0-1}\left(\frac{k_0+1}{k_0+2}-\frac{k+1}{k+2}\right).
\end{eqnarray}
Since $M_{n_j}(\a)\to\infty$, there exists $j$  so that for $\tilde{n}:=n_j\ge k_0$
\begin{equation}\label{eq:interm}
\a\frac{k_0+1}{k_0+2}M_{\tilde{n}}(\a)-\a\sum_{k=0}^{k_0-1}\left(\frac{k_0+1}{k_0+2}-\frac{k+1}{k+2}\right)>M_{\tilde{n}}(\a).
\end{equation}
This gives $p_{\tilde{n}+1}(\a)>M_{\tilde{n}}(\a)$ and hence
 $M_{\tilde{n}+1}(\a)>M_{\tilde{n}}(\a)$.

We prove now by induction that $(M_k(\a))$ is strictly increasing for $k\ge \tilde{n}$. We have just established the start of the induction proof.

Assume that for some $k\in\N$
$$
M_{\tilde{n}}(\a)< M_{\tilde{n}+1}(\a)<\cdots < M_{\tilde{n}+k}(\a).
$$
By \eqref{eq:fundl} we have
\begin{eqnarray*}
p_{\tilde{n}+k+1}(\a)  &>&  \a\frac{k_0+1}{k_0+2}M_{\tilde{n}+k}(\a)-\a\sum_{k=0}^{k_0-1}\left(\frac{k_0+1}{k_0+2}-\frac{k+1}{k+2}\right)\\
&=& M_{\tilde{n}+k}(\a)\left[\a\frac{k_0+1}{k_0+2} - \frac{\a}{M_{\tilde{n}+k}(\a)}
\sum_{k=0}^{k_0-1}\left(\frac{k_0+1}{k_0+2}-\frac{k+1}{k+2}\right)\right]\\
&>&M_{\tilde{n}+k}(\a)\left[\a\frac{k_0+1}{k_0+2} - \frac{\a}{M_{\tilde{n}}(\a)}
\sum_{k=0}^{k_0-1}\left(\frac{k_0+1}{k_0+2}-\frac{k+1}{k+2}\right)\right]\\
&>&M_{\tilde{n}+k}(\a),
 \end{eqnarray*}
where the last inequality follows from \eqref{eq:interm}.
\end{proof}

\begin{proof} [Proof of Proposition~\ref{thm:p-bounds}-(vii)]

We will only prove  $\lim_{n\to\infty} \frac{p_{n+1}(\a)}{p_n(\a)}=1$, then  $\lim_{n\to\infty} \root{n}\of {p_n(\a)}=1$ follows from \cite[Theorem 3.37]{R}.

For every $\a>0$, by \eqref{eq:pnp_pn}, 
 \begin{equation}\label{eq:liratio}
\liminf_{n\to\infty}\frac{p_{n+1}(\a)}{p_n(\a)}\geq1.
\end{equation}

For $0<\a\leq1$, we already know from Proposition \ref{thm:p-bounds}-$(v)$ that 
$\frac{p_{n+1}(\a)}{p_n(\a)}<1$ and then  
 $$
\lim_{n\to\infty}\frac{p_{n+1}(\a)}{p_n(\a)}=1.
$$

By \eqref{eq:mean2} and  \eqref{eq:mean1} we get
  \begin{equation}\label{eq:Mto1}
\frac{M_n(\a)}{M_{n-1}(\a)}=\frac{n}{n+1}+\frac{p_n(\a)}{(n+1)M_{n-1}(\a)}\leq \frac{n}{n+1}+\frac{\a}{(n+1)}.
\end{equation}

For $\a\geq2$, by  \eqref{eq:mean1} and Property $(v)$ in Proposition~\ref{thm:p-bounds} we have 
\begin{equation}\label{eq:pleqM}
p_{k}(\a)\leq p_{n+1}(\a)\leq\a M_n (\a) \qquad\text{for every $k\leq n$}.
\end{equation}

By  \eqref{eq:mean1} and \eqref{eq:fund}, 
\begin{eqnarray*}
\frac{p_{n+2}(\a)}{p_{n+1}(\a)}&\leq&\frac{\a M_{n+1}(\a)}{\a\frac{k_0+1}{k_0+2}M_n(\a)-\frac{\a}{n+1}\sum_{k=0}^{k_0-1}\left(\frac{k_0+1}{k_0+2}-\frac{k+1}{k+2}\right)p_{n-k}(\a)}
\\
&=&\frac{ M_{n+1}(\a)}{ M_{n}(\a)} \cdot\frac1{\frac{k_0+1}{k_0+2}-\frac{1}{n+1}\sum_{k=0}^{k_0-1}\left(\frac{k_0+1}{k_0+2}-\frac{k+1}{k+2}\right)\frac{p_{n-k}(\a)}{M_n(\a)}}\,.
\end{eqnarray*}
For a given $k_0$ the sum in the  denominator is bounded because of \eqref{eq:pleqM}. Then we obtain, using 
\eqref{eq:Mto1}, 
$$
\limsup_{n\to\infty}\frac{p_{n+2}(\a)}{p_{n+1}(\a)}\leq\frac{k_0+2}{k_0+1}.
$$
Since $k_0$ is any integer, along with \eqref{eq:liratio}, this proves that 
$$
\lim_{n\to\infty}\frac{p_{n+1}(\a)}{p_{n}(\a)}=1.
$$

 For $1<\a\leq2$ we use \eqref{eq:fundl} instead of \eqref{eq:fund}, to obtain 
\begin{eqnarray} 
\nonumber\frac{p_{n+2}(\a)}{p_{n+1}(\a)}&\leq&\frac{\a M_{n+1}(\a)}{ \a\frac{k_0+1}{k_0+2}M_n(\a)-\a\sum_{k=0}^{k_0-1}\left(\frac{k_0+1}{k_0+2}-\frac{k+1}{k+2}\right)}
\\\nonumber
&=&\frac{ M_{n+1}(\a)}{ M_{n}(\a)}\cdot\frac1{\frac{k_0+1}{k_0+2}-\frac{1}{M_{n}(\a)}\sum_{k=0}^{k_0-1}\left(\frac{k_0+1}{k_0+2}-\frac{k+1}{k+2}\right)}\,.
\end{eqnarray}
Again, for a fixed $k_0$ the sum in the denominator is bounded and since $M_{n}(\a)\to\infty$ we conclude as above.
\end{proof}

\begin{proof}[Proof of Theorem \ref{thm:alt}]
In the case $0\le\a\le1$, Equation \eqref{eq:ratbdd}  follows immediately from  Proposition~\ref{thm:p-bounds}-$(v)$.

For $2\le\a$, using that  $p_n(\a)$ is increasing by  Proposition~\ref{thm:p-bounds}-$(v)$ we get
	$$
	\frac{p_{n+1}(\a)}{p_n(\a)}=\frac{\a}{n+1}\sum_{k=0}^n \frac{k+1}{k+2} \frac{p_{n-k}(\a)}{p_n(\a)}
	\le \frac{\a}{n+1}\sum_{k=0}^n \frac{k+1}{k+2}\le \a.
	$$

Finally, in the case $1<\a<2$, Equation \eqref{eq:ratbdd} can be obtained combining \eqref{eq:Mto1} with  \eqref{eq:mean1}: 
$$\frac{p_{n+1}(\a)}{p_{n}(\a)}\leq 2\frac{ M_{n}(\a)}{ M_{n-1}(\a)}\leq 2\frac {n+\a}{n+1}\leq 2\a,\qquad n\ge1,$$
while $\frac{p_{1}(\a)}{p_{0}(\a)}=\a/2<2\a$.

Now, for $\a> 0, s>0$ and $n\ge 1$ we have
	$$
	p_{n+1}(\a)\frac{s^n}{n!}\le p_n(\a)\frac{s^{n-1}}{(n-1)!}
	$$
	if and only if
	$$
	\frac{p_{n+1}(\a)}{np_n(\a)}\le \frac{1}{s}.
	$$
	Since the left-hand side is $\le \widehat\a/n$ by Equation \eqref{eq:ratbdd},  we see that the power series \eqref{eq:ent1} satisfies the  Alternating Series Test for
	$n\ge \widehat\a s$.
\end{proof}

\begin{rem}
{\rm Property $(v)$ in Proposition \ref{thm:p-bounds} does not consider the case $1<\a<2$. 
From numerical calculations it seems true that $(p_n(\a))_{n\ge 1}$ is increasing  whenever $\a\ge 4/3$, which is when $p_1(\a)\le p_2(\a)$. For $1<\a<4/3$,  $p_n(\a)$ is decreasing for $0\le n\le n_0(\a)$ and increasing for $n_0(\a)\le n$, where $n_0(\a)\in\N$ is decreasing in $\a$. However, we have not been able to prove this.}
\end{rem}

\section{Proof of Theorem \ref{thm:final}}\label{sec:propphi}
In this section we prove several properties of the family of functions  $\varphi_\a$, that were listed in Theorem \ref{thm:final}

In the proof of Theorem~\ref{thm:final}-$(i,\,iii)$ we need the following lemma. The proof is left as an exercise.

\begin{lem}\label{thm:entire} Let
$$
f_j(s)=\sum_{n=0}^\infty a_{j,n}s^n,\quad f(s)=\sum_{n=0}^\infty a_{n}s^n, \quad s\in\C
$$
be power series of entire functions $f_j,j\in\N$ and $f$. Assume that for all $n\ge 0$
\begin{equation*} 
\lim_{j\to\infty} a_{j,n}=a_n,\quad |a_{j,n}|\le c_n,
\end{equation*}
where $\sum c_nR^n<\infty$ for all $R>0$.

Then $\lim_{j\to\infty} f_j(s)=f(s)$ uniformly for $s$ in compact subsets of the complex plane.
\end{lem}

\begin{proof}[Proof of Theorem~\ref{thm:final}-(i,\ ii)]
 We use that 
$$
\lim_{\a\to 0}\frac{p_{n+1}(\a)}{\a}=\frac{1}{n+2},\qquad \left|\frac{p_{n+1}(\a)}{\a}\right|\le 1,\quad 0<|\a|\le 1.
$$ 
The first assertion follows from Proposition \ref{thm:p-bounds}-$(i)$, and the second assertion follows from Proposition \ref{thm:p-bounds}-$(ii,\,iii)$ together with \eqref{eq:poly-rec}. 

Lemma~\ref{thm:entire} now shows that
$$
\lim_{\a\to 0} \frac{\varphi_\a(s)}{\a e^\a}=\sum_{n=0}^\infty\frac{(-1)^n}{n+2}\frac{s^n}{n!}
$$
uniformly for $s$ in compact subsets of $\C$. It is easy to see that the sum of this power series is equal to the function
$$
w(s)=\left\{ \begin{array}{ll}\displaystyle \frac{1-(1+s)e^{-s}}{s^2},\, &\mbox{when}\, s\neq 0, \\
\frac12,\,&\mbox{when}\, s=0.
\end{array}\right.
$$
The zeros  of $w$ are given by $1+s=e^s$, which has no real solutions different from 0. The equation $1+s=e^s$ has countably many non-real solutions, which can be  given using the branches of the Lambert $W$ function, available in Maple. For each $k\in\Z$ the $k$'th branch is denoted $W(k,z)$ and satisfies $W(k,z)\exp(W(k,z))=z$. It follows that the solutions to $1+s=e^s$ are $s=\xi_k=-W(k,-1/e)-1, k\in\Z$, but $\xi_{-1}=\xi_0=0$ and the other values given are calculated in Maple. 

Assertion $(ii)$ now  follows from Hurwitz' Theorem. 
\end{proof}
\begin{proof}[Proof of Theorem~\ref{thm:final}-$(iii,\,iv)$]
For simplicity we introduce
\begin{equation}\label{eq:phitil}
\widetilde{\varphi}_\a(s)=\exp(-\a)\varphi_\a(s)=\sum_{n=0}^\infty (-1)^np_{n+1}(\a)\frac{s^n}{n!}.
\end{equation}
Note that 
$$
\lim_{|\a|\to\infty}\frac{p_{n}(\a)}{\a^n}=\frac{1}{2^nn!}, \qquad \left|\frac{p_{n}(\a)}{\a^n}\right|\le 1,\quad \mbox{for}\ 1\le |\a|,
$$
by Proposition \ref{thm:p-bounds}-$(i,\, ii,\  iii)$. Lemma~\ref{thm:entire} therefore shows that
$$
\lim_{|\a|\to\infty}\frac{\widetilde{\varphi}_\a(s/\a)}{\a}=\sum_{n=0}^\infty(-1)^n\frac{s^n}{2^{n+1}n! (n+1)!}
$$
uniformly for $s$ in compact subsets of the complex plane.

The Bessel function of order 1 is defined by the series
$$
J_1(z)=\frac{z}{2}\sum_{n=0}^\infty (-1)^n\frac{(z/2)^{2n}}{n!(n+1)!}
$$
and hence
\begin{equation}\label{eq:zer}
\lim_{|\a|\to\infty}\frac{\widetilde{\varphi}_\a(s/\a)}{\a}=\frac{J_1(\sqrt{2s})}{\sqrt{2s}}.
\end{equation}
The zeros of $J_1$ are all real and simple and equal to $0,\pm j_1,\pm j_2,\ldots$, where $0<j_1<j_2<\ldots$ is a well-known sequence of positive numbers tending to infinity.

The zeros of the right-hand side of \eqref{eq:zer} are $j_k^2/2$.
 In a sufficiently small disc $D_k$ centered at $j_k^2/2$, $\widetilde{\varphi}_{\a}(s/\a)$ has a unique zero $s_k(\a)$, when $|\a|$ is sufficiently large. It is simple and we have
$$
\lim_{|\a|\to\infty} \a s_k(\a)=\frac{j_k^2}{2}.
$$
This is according to a theorem of Hurwitz. If $\a>0$ the complex zeros of $\varphi_{\a}$ must appear in conjugate pairs, and therefore $s_k(\a)$ must be real and hence positive for otherwise $D_k$ would contain two zeros, when $\a>0$ is sufficiently large. 
\end{proof}
\begin{proof}[Proof of Theorem~\ref{thm:final}-(v)] The order and type does not change when we multiply an entire function by a constant, then we may work with $\widetilde{\varphi}_\a$ as in \eqref{eq:phitil}. Defining 
\begin{equation*} 
c_n=\frac{p_{n+1}(\a)}{n!},
\end{equation*}
it is known, cf. e.g. \cite{Bo}, that the order $\rho$ of $\varphi_\a$ is
$$
\rho=\limsup_{n\to\infty}\frac{\log n}{\log(1/\root{n}\of{c_n})},
$$ 
but since
$$
\frac{1}{\root{n}\of{c_n}}=\frac{\root{n}\of{n!}}{\root{n}\of{p_{n+1}(\a)}}\sim\frac{n}{e}
$$
by Proposition~\ref{thm:p-bounds}-$(vii)$ and Stirling's formula, we get
$$
\frac{\log n}{\log(1/\root{n}\of{c_n})}=\left(1+\frac{\log(e/(n\root{n}\of{c_n}))-1}{\log n}\right)^{-1},
$$
which converges to 1, hence $\rho=1$.

The type $\tau$ is given by
$$
\tau=\frac{1}{e}\limsup_{n\to\infty}(n\root{n}\of{c_n}),
$$
but since $\lim(n\root{n}\of{c_n})=e$, we get $\tau=1$.
\end{proof}

\section{Appendix}
In this appendix we add a few remarks that came up after the submission of this work.
\begin{rem}\label{thm:referee} {\rm 
A referee has kindly pointed out that the coefficients $c_{n,k}$ of the polynomials $p_n(\a)$ can be expressed by the following formula
\begin{equation}\label{eq:referee1}
c_{n,k}=(-1)^{n-k}\sum_{m=1}^k (-1)^m\frac{s(n+m,m)}{(n+m)!(k-m)!},\quad n\ge k\ge 1,
\end{equation}
where the $s(p,m)$ are the Stirling numbers of the first kind defined by 
$$
t(t-1)\cdots(t-p+1)=\sum_{m=0}^p s(p,m)t^m,\quad p\ge 1,
$$
 $s(0,0):=1$, see \cite[p.278]{Ch}. Note that $s(p,0)=0$ for $p\ge 1$, so in \eqref{eq:referee1} one may sum from $m=0$ as well.
To see \eqref{eq:referee1} we use the formula
$$
B_n(a_1,\ldots,a_n)=\sum_{k=1}^n B_{n,k}(a_1,\ldots,a_{n-k+1}),
$$
where the partial Bell partition polynomials $B_{n,k}$ are defined as
$$
B_{n,k}(a_1,\ldots,a_{n-k+1})=\sum_{J(n,k)}\frac{n!}{j_1!\cdots j_{n-k+1}!}\prod_{m=1}^{n-k+1}\left(\frac{a_m}{m!}\right)^{j_m},
$$
cf. \cite[Section 11.2]{Ch}. The sum is over the set $J(n,k)$ of all integers $j_1,\ldots,j_{n-k+1}\ge 0$ satisfying
$$
j_1+\cdots +j_{n-k+1}=k,\quad j_1+2j_2+\cdots +(n-k+1)j_{n-k+1}=n.
$$
In the special case $a_k=(-1)^k\a k!/(k+1),k=1,\ldots,n$ we then get
\begin{eqnarray*}
&&B_{n,k}\left(-\a\frac{1!}{2},\a\frac{2!}{3},\ldots,(-1)^{n-k+1}\a\frac{(n-k+1)!}{n-k+2}\right)\\
&=&\sum_{J(n,k)}\frac{n!}{j_1!\cdots j_{n-k+1}!}\prod_{m=1}^{n-k+1}\left(\frac{(-1)^m\a}{m+1}\right)^{j_m}\\
&=&\a^k(-1)^nB_{n,k}\left(\frac{1!}{2},\frac{2!}{3},\ldots,\frac{(n-k+1)!}{n-k+2}\right).
\end{eqnarray*}
In \cite[Theorem 1]{Q} one finds the evaluation
\begin{equation}\label{eq:Bnk}
B_{n,k}\left(\frac{1!}{2},\frac{2!}{3},\ldots,\frac{(n-k+1)!}{n-k+2}\right)=
(-1)^{n-k}n!\sum_{m=1}^k\frac{(-1)^ms(n+m,m)}{(n+m)!(k-m)!},
\end{equation}
and hence by \eqref{eq:pnBn}
\begin{eqnarray*}
p_n(\a)=\frac{(-1)^n}{n!}\sum_{k=1}^n \a^k(-1)^n(-1)^{n-k}n!\sum_{m=1}^k\frac{(-1)^ms(n+m,m)}{(n+m)!(k-m)!},
\end{eqnarray*}
and finally one obtains \eqref{eq:referee1}.

We observe that, from \eqref{eq:referee1}, it is possible to deduce the explicit formula for $c_{n,1}$ given in Proposition~\ref{thm:p-bounds}-$(i)$, using that $(-1)^ns(n+1,1)=n!$, and also, since  $s(n,2)=(-1)^n(n-1)! H_{n-1}$ with $H_n=1+1/2+...+1/n$ being the $n^{th}$ harmonic number,  to obtain the following  formula for $c_{n,2}$: $$
c_{n,2}=\frac{H_{n+1}}{n+2}-\frac{1}{n+1}.$$ 
Further formulas can be obtained in terms of generalized harmonic numbers but they become increasingly more complicated.

Also the explicit formula for $c_{n,n}$ given in Proposition~\ref{thm:p-bounds}-$(i)$ can be obtained, by using  \eqref{eq:Bnk} and the definition of $B_{n,n}$:
$$c_{n,n}=\sum_{m=1}^n\frac{(-1)^m s(n+m,m)}{(n+m)!(n-m)!}=\frac1{n!}B_{n,n}\left(\frac12\right)=\frac1{n!}\left(\frac12\right)^n.$$

}
\end{rem}

\begin{rem}\label{thm:n=-1}{\rm Alan Sokal asked the first author if Theorem~\ref{thm:Haus} can be replaced by the stronger statement that $(p_n(\a))_{n\ge 0}$ is a Hausdorff moment sequence when $0\le \a\le 1$. The answer is yes, but the reader is warned that Equations \eqref{eq:Haus1} and \eqref{eq:Haus2} do not hold for $n=-1$.

In fact, if $0<\a<1$ we get for $n=-1$ 
\begin{eqnarray*}
&&\frac{e^{-\a}}{\pi}\int_0^1 (x/(1-x))^{\a x}\sin(\a\pi x) x^{-1}\,dx =e^{-\a}\lim_{x\to 0^+}f_\a(x)\\
&=&e^{-\a}(e^{\a}-1)=1-e^{-\a}<1=p_0(\a),
\end{eqnarray*}
where we have used \eqref{eq:St2} and \eqref{eq:fa(0)}, and there is a similar calculation  in case $\a=1$. Using the Hausdorff moment sequence $(\delta_{n0})_{n\geq0}=(1,0,0,0,\ldots)$ we find for $0<\a<1$
\begin{equation*} 
p_n(\a)=e^{-\a}\delta_{n0}+\frac{e^{-\a}}{\pi}\int_0^1 (x/(1-x))^{\a x}\frac{\sin(\a\pi x)}{x}x^n\,dx,\quad n\ge 0,
\end{equation*}
showing that $(p_n(\a))_{n\ge 0}$ is a Hausdorff moment sequence when $0<\a<1$.
We  similarly get $p_n(0)=\delta_{n0}$ and
$$
p_n(1)=e^{-1}\delta_{n0}+e^{-1}+\frac{e^{-1}}{\pi}\int_0^1 (x/(1-x))^{x} \frac{\sin(\pi x)}{x}x^n\,dx,\quad n\ge 0.
$$ 
}
\end{rem}

\FloatBarrier

\par\bigskip

\begin{center}
{\bf Acknowledgments}
\end{center}

This work was initiated during a visit of the first author to the Department of Mathematics at the  University of S\~ao Paulo in S\~ao Carlos, Brazil, in March 2018. He wants to thank the Department for generous support and hospitality during his stay.

The second author was  supported by: grant $\#$303447/2017-6, CNPq/Brazil. 
	
The third author was supported by: grant $\#$2016/09906-0, S\~ao Paulo Research Foundation (FAPESP).

The authors thank a referee for useful references leading in particular to Remark~\ref{thm:referee}.

\noindent
Christian Berg\\
Department of Mathematical Sciences, University of Copenhagen\\
Universitetsparken 5, DK-2100, Denmark\\
e-mail: {\tt{berg@math.ku.dk}}

\vspace{0.4cm}
\noindent
Eugenio Massa\\
Departamento de Matem{\'a}tica, ICMC-USP-S{\~a}o  Carlos\\
Caixa Postal 668, 13560-970 S{\~a}o Carlos SP, Brazil\\
e-mail: {\tt{eug.massa@gmail.com}} 

\vspace{0.4cm}
\noindent
Ana P. Peron\\
Departamento de Matem{\'a}tica, ICMC-USP-S{\~a}o  Carlos\\
Caixa Postal 668, 13560-970 S{\~a}o Carlos SP, Brazil\\
e-mail: {\tt{apperon@icmc.usp.br}} 

\end{document}